\documentclass[11pt,reqno]{amsart}
\usepackage{amsmath, amscd}
\usepackage{amssymb}
\usepackage{amsthm}
\usepackage[all]{xy}
\usepackage{lscape}
\usepackage[vmargin=1in, hmargin=1.25in]{geometry}
\usepackage{longtable}
\usepackage[backref=page, bookmarks, bookmarksdepth=2, colorlinks=true, linkcolor=blue, citecolor=blue, urlcolor=blue]{hyperref}
\usepackage[T1]{fontenc}
\usepackage{lscape}
\usepackage{enumitem}

\makeatletter

\renewcommand\subsubsection{\@secnumfont}{\bfseries}%
\renewcommand\subsubsection{\@startsection{subsubsection}{3}
  \z@{.5\linespacing\@plus.7\linespacing}{-.5em}%
  {\normalfont\bfseries}}
  
  \makeatother

\numberwithin{equation}{section}

\renewcommand*{\backref}[1]{}

\theoremstyle{plain}
\newtheorem{thm}{Theorem}[section]
\newtheorem{cor}[thm]{Corollary}
\newtheorem{prop}[thm]{Proposition}
\newtheorem{lem}[thm]{Lemma}
\newtheorem{conj}[thm]{Conjecture}
\newtheorem{defn}[thm]{Definition}
\newtheorem{ex}[thm]{Example}
\newtheorem{notation}[thm]{Notation}

\newtheorem{taggedtheoremx}{Theorem}
\newenvironment{taggedtheorem}[1]
 {\renewcommand\thetaggedtheoremx{#1}\taggedtheoremx}
 {\endtaggedtheoremx}
 
\theoremstyle{remark}
\newtheorem{remark}[thm]{Remark}

\theoremstyle{plain}
\DeclareMathOperator{\Q}{\mathbb{Q}}
\DeclareMathOperator{\M}{\textup{Mod}}
\DeclareMathOperator{\fH}{\mathfrak{H}}
\newcommand{\defeq}{\mathrel{\mathop:}=}

\newcommand{\xdownarrow}[1]{%
  {\left\downarrow\vbox to #1{}\right.\kern-\nulldelimiterspace}
}
\let\oldproofname=\proofname
\renewcommand{\proofname}{\rm\bf{\oldproofname}}

\title[Prym Representations and Twisted Cohomology]{Prym Representations and Twisted Cohomology of the Mapping Class Group with Level Structures}
\author{Xiyan Zhong}

\begin{document}

\begin{abstract}
We compute the twisted cohomology of the mapping class group with level structures, with coefficients in the $r$-tensor powers of the Prym representations for any positive integer $r$. When $r\ge 2$, we show that the cohomology exhibits instability for large genus, whereas it remains stable for $r=0$ or $r=1$. As a corollary, we prove that the symplectic Prym representation associated with any finite abelian regular cover of a non-closed finite-type surface is infinitesimally rigid.
\end{abstract}

\maketitle

\section{Introduction}

The (pure) mapping class group $\M(S)$ of an oriented surface $S$ is the group of isotopy classes of orientation-preserving homeomorphisms that fix the boundary pointwise, including punctures and boundary components. For each integer $\ell\ge 2$, the level-$\ell$ mapping class group of $S$ is the finite-index subgroup of $\M(S)$ consisting of elements that act trivially on $H_1(S;\mathbb{Z}/\ell)$. An important family of representations defined on these level subgroups are the Prym representations, first studied by Looijenga \cite{prym}.

The main result of this paper is the computation of the cohomology of the level-$\ell$ mapping class group with twisted coefficients given by the $r$-tensor powers of the Prym representation, for arbitrary $r\ge 1$, in a range where the genus of $S$ is sufficiently large. Our calculations uncover a new phenomenon: when $r\ge 2$, this twisted cohomology exhibits instability with respect to the genus of $S$, in contrast to the known cohomological stability of mapping class groups and the stable behavior in the case $r=1$. 

As another consequence of our work, we prove that the Prym representation is infinitesimally rigid when the genus of $S$ is at least $41$.

The central technical innovation of this paper is the construction of a refined bordification of the moduli space of Riemann surfaces with level structures in Section 4. The boundary strata of this space are described using the theory of orbit configuration spaces, which allows us to compute its cohomology. By analyzing a specific Leray spectral sequence associated with this bordified moduli space and applying the theory of mixed Hodge structures, we derive our main results.

Since the statement of infinitesimal rigidity for the Prym representations is more straightforward, we begin by presenting this result.

\subsection{Rigidity Theorem}

Let $\widetilde{S}\to S$ be a regular cover whose deck transformation group is a finite abelian group $A$. Let $\M(S,A)$ denote the subgroup of $\M(S)$ consisting of mapping classes that lift to homeomorphisms of $\widetilde{S}$. The induced action 
$$\M(S,A)\to \text{Aut}(H^1(\widetilde{S};\Q))$$ defines the associated \textbf{Prym representation}. When $S$ is closed, Looijenga determined the image of the Prym representation up to finite index (\cite[Theorem 2.5]{prym}).

When $S$ is not closed, let $\widehat{S}$ denote the closed surface obtained by gluing disks to all boundary components and filling in all punctures of $\widetilde{S}$. Any lift of a homeomorphism of $S$ to $\widetilde{S}$ preserves homology classes of loops around boundary components of $\widetilde{S}$, and commutes with the deck group $A$. It is therefore natural to consider a refined version of the Prym representation:
$$\Phi: \M(S,A)\to \text{Aut}(H^1(\widehat{S};\mathbb{R}))^A.$$ We refer to this as the \textbf{symplectic Prym representation}. We work over $\mathbb{R}$ rather than $\Q$ because, when the image of $\Phi$ is a Lie group, this allows us to study its rigidity in the following sense.

Let $\Gamma$ be a finitely generated group and $G$ a Lie group. A homomorphism $\Phi:\Gamma \to G$ is called \textbf{infinitesimally rigid} if $$H^1(\Gamma;\mathfrak{g})=0,$$ where $\mathfrak{g}$ is the Lie algebra of $G$, and $\Gamma$ acts on $\mathfrak{g}$ via the composition of $\Phi$ and the adjoint representation  $\textup{Ad}:G\to \textup{Aut}(\mathfrak{g})$. Intuitively, infinitesimal rigidity means that $\Phi$ admits no nontrivial first-order deformations. We establish this kind of rigidity for the symplectic Prym representation:

\begin{taggedtheorem}{A}\label{main1}
Let $S$ be a non-closed surface of genus $g\ge 41$. For any finite abelian cover $\widetilde{S}\to S$ with deck transformation group $A$, the associated symplectic Prym representation 
$$\Phi: \M(S,A)\to \text{Aut}(H^1(\widehat{S};\mathbb{R}))^A$$ is infinitesimally rigid.
\end{taggedtheorem}

\begin{remark}\label{rmk1.17}
 Infinitesimal rigidity implies other kinds of rigidity as follows:
\begin{enumerate}
\item Weil (\cite[p.~152]{weil}) proved that an infinitesimally rigid representation is \textbf{locally rigid}. A representation $\Phi$ is locally rigid
if $[\Phi]$ is an isolated point in $Hom(\Gamma,G)/G$, i.e. any $\Phi'\in Hom(\Gamma,G)$ sufficiently close to $\Phi$ is conjugate to $\Phi$. The converse is false: locally rigidity does not imply infinitesimal rigidity (see e.g \cite[(2.10.4)]{var}).
\item For local systems on quasi-projective varieties, infinitesimal rigidity also implies \textbf{cohomological rigidity} when the Lie group $G$ is semi-simple. See the definition of cohomological rigidity in \cite{esna} and \cite{klev}.
\end{enumerate} 
\end{remark}

\subsection{Computations of Twisted Cohomology} To state our results on the twisted cohomology of mapping class groups with level structures, we introduce the following setup. Let $\Sigma_{g,p}^b$ denote a genus-$g$ surface with $p$ (ordered) punctures and $b$ boundary components. We omit $p$ or $b$ when it is $0$. Let $\M_{g,p}^b$ denote the mapping class group of $\Sigma_{g,p}^b$, and define the level-$\ell$ mapping class group of $\Sigma_{g,p}^b$ by
$$\M_{g,p}^b(\ell)=\text{Ker}(\M_{g,p}^b\to \text{Aut}(H_1(\Sigma_{g,p}^b;\mathbb{Z}/\ell)).$$
We remark that when $p+b\ge 2$, the kernel of the action of $\M_{g,p}^b$ on $H_1(\Sigma_{g};\mathbb{Z}/\ell)$ strictly contains $\M_{g,p}^b(\ell)$.

Let $\mathcal{D}=H_1(\Sigma_g;\mathbb{Z}/\ell)$. Consider the composition of group homomorphisms
$$\pi_1(\Sigma_{g,p}^b)\to H_1(\Sigma_{g,p}^b;\mathbb{Z}/\ell)\to H_1(\Sigma_g;\mathbb{Z}/\ell),$$
where the second map is induced by the embedding $\Sigma_{g,p}^b\to\Sigma_g$ obtained by gluing disks to all boundary components and filling in all punctures of $\Sigma_{g,p}^b$. This group homomorphism induces a regular cover of $\Sigma_{g,p}^b$ with deck transformation group $\mathcal{D}$, which we denote by $S_{\mathcal{D}}\to \Sigma_{g,p}^b$. By covering space theory, elements of $\M_{g,p}^b(\ell)$ lift to homeomorphisms of $S_{\mathcal{D}}$ that fix all punctures and boundary components pointwise. Define $$\fH_{g,p}^b(\ell;\Q)\defeq H^1(S_{\mathcal{D}};\Q),$$ which is a $\M_{g,p}^b(\ell)$-module. We refer to $\fH_{g,p}^b(\ell;\Q)$ as the Prym representation of $\M_{g,p}^b(\ell)$. Our central result is the computation of the cohomology
$$H^{k}(\M_{g,p}^b(\ell);\fH_{g,p}^b(\ell;\Q)^{\otimes r})$$
in a range where the genus $g$ is sufficiently large relative to $k$, for arbitrary $r\ge 1$:

\begin{taggedtheorem}{B}\label{main2}
Let $A'_{r}(\ell)^{\bullet}$ be the graded $\Q$-vector space defined in \eqref{mainsubspace} of Section 5. The symmetric group $S_r$ acts naturally on $A'_{r}(\ell)^{\bullet}$ by definition, and also on $\fH_{g,p}^b(\ell;\Q)^{\otimes r}$ by permuting tensor factors. There is a graded $S_r$-equivariant map of $H^{\bullet}(\M_{g,p}^b(\ell);\mathbb{Q})$-modules:
$$\begin{aligned}H^{\bullet}(\M_{g,p}^b(\ell);\mathbb{Q})  \otimes A'_{r}(\ell)^{\bullet} \to H^{\bullet-r}(\M_{g,p}^b(\ell);\fH_{g,p}^b(\ell;\Q)^{\otimes r})\end{aligned},$$
which is an isomorphism in degrees $k$ such that $g\ge 2k^2+7k+2$.
\end{taggedtheorem}
Here $A'_{r}(\ell)^{\bullet}$ is a subspace of the rational cohomology ring of a moduli space denoted $\mathcal{C}_{g,r}(\ell)$. This moduli space contains the moduli space $\mathcal{M}_{g,r}(\ell)$ of Riemann surfaces homeomorphic to $\Sigma_{g,r}$ with a level-$\ell$ structure as an open subvariety. In section 4, we construct $\mathcal{C}_{g,r}(\ell)$ explicitly, which is a bordification of $\mathcal{M}_{g,r}(\ell)$ with boundary strata indexed by orbit configuration spaces. We compute in Theorem \ref{coh2} the rational cohomology of $\mathcal{C}_{g,r}(\ell)$ for sufficiently large $g$. As defining $A'_{r}(\ell)^{\bullet}$ requires additional notation from Section 4, we postpone its formal definition to Section 5.

Theorem \ref{main2} enables us to compute the rational dimension of $$H^{k-r}(\M_{g,p}^b(\ell);\fH_{g,p}^b(\ell;\Q)^{\otimes r})$$ explicitly, thanks to Putman's theorem (\cite[Theorem A]{AndyStable}), which states that
$$H^k(\M_{g,p}^b(\ell);\Q) \cong H^k(\M_{g,p}^b; \Q) \text{ for } g\ge 2k^2 + 7k + 2.$$ 
By the Madsen-Weiss Theorem (\cite[Theorem 1.1.1]{Madsen}) and Looijenga's extension to punctured surfaces (\cite[Prop.~2.2]{LooiStable}), this implies that 
$$H^{\bullet}(\M_{g,p}^b(\ell);\mathbb{Q})\cong \Q[e_1,\cdots,e_p]\otimes \Q[\kappa_1,\kappa_2,\cdots],$$
where $\kappa_i\in H^{2i}(\M_{g,p}^b;\Q)$ ($i\ge 1$) are the Miller-Morita-Mumford classes (\cite{miller}, \cite{morita}), and $e_j\in H^2(\M_{g,p}^b;\mathbb{Q})$ is the Euler class associated with the central extension
$$ 1\to \mathbb{Z}\to \M_{g,p-1}^{b+1} \to \M_{g,p}^b\to 1$$
obtained by blowing up the $j$-th puncture of $\Sigma_{g,p}^b$ to a boundary component. The central $\mathbb{Z}$ is generated by the Dehn twist around this boundary component. See also \cite[Sec.~7]{morita} for an alternative definition of $e_j$.

\begin{remark}
From the definition of $(A'_{r,\ell})^{\bullet}$ (see \eqref{mainsubspace}), it follows that all its elements are of even degree. Therefore, Theorem \ref{main2} implies that for odd $k$, $$H^{k-r}(\M_{g,p}^b(\ell);\fH_{g,p}^b(\ell;\Q)^{\otimes r}=0, \text{ for }g\ge 2k^2+7k+2.$$ In particular, Theorem \ref{main1} follows as a corollary of the special case $k=3, r=2$.
\end{remark}

\subsection{Instability and Historical Remarks} Computations via Theorem \ref{main2} show that for $r\ge 2$ and an even number $k$,
$$H^{k-r}(\M_{g,p}^b(\ell);\fH_{g,p}^b(\ell;\Q)^{\otimes r})$$ is independent of $b$, but depends on $g$, $p$, and $\ell$ for $r\ge 2$. This instability arises from the fact that our calculations involve orbit configuration spaces, which depend on the deck transformation group $\mathcal{D}=H_1(\Sigma_g;\mathbb{Z}/\ell)$.

This instability also helps us identify relationship with some other twisted cohomology groups related to mapping class groups, which we now recall.

The mapping class group $\M_{g,p}^b$ acts non-trivially on $H^1(\Sigma_{g,p}^b;\Q)$. For any positive integer $r$, the twisted cohomology
\begin{equation}\label{hmh} H^{k-r}(\M_{g,p}^b;H^1(\Sigma_{g,p}^b;\Q)^{\otimes r})\end{equation}
has been studied by Looijenga \cite{LooiStable}, Kawazumi \cite{kawa}, Randal-Williams \cite{rw1}, and Putman \cite{AndyStable}. 

In the case $p = b = 0$, this twisted cohomology \eqref{hmh} was computed by Looijenga (\cite[Corollary 3.3]{LooiStable}) in 1996 for $g \ge \frac{3}{2}k + 1$. He constructed an $S_r$-equivariant map analogous to the one in Theorem \ref{main2}, which allowed him to further compute the cohomology of $\M_g$ with coefficients in any irreducible representation of $Sp_{2g}(\mathbb{C})$.

In 2008, Kawazumi (\cite[Theorem 1.A]{kawa}) computed the twisted cohomology \eqref{hmh} for $\Sigma_g^b$ with $b \ge 1$, also in the same stable range, but over $\mathbb{Z}$, using a different inductive approach. Kawazumi’s results exhibit stability with respect to both $g$ and $b$.

In 2018, Randal-Williams (\cite[Appx.~B]{rw1}) identified \eqref{hmh} as a specific $S_r$-module in the cases of $\Sigma_g$ and $\Sigma_g^1$. Later, in 2020, Randal-Williams and Kupers (\cite[Theorem 3.15]{rwk}) incorporated additional structure into the coefficients, allowing them to compute the cohomology of $\M_g^1$ with coefficients in any algebraic representation of $Sp_{2g}$.

\begin{remark}
Randal-Williams suggested that it might be possible to generalize the results of Theorem \ref{main2} from coefficients in tensor powers to coefficients in Schur functors. We plan to explore this direction in future work.
\end{remark}

For completeness, we derive from the work of Looijenga and Kawazumi the remaining case of \eqref{hmh} where $b = 0$ and $p \ge 1$, and state the general result for any non-closed surface $\Sigma_{g,p}^b$ in Theorem \ref{thmA}.

The twisted cohomology \eqref{hmh} is related to the cohomology $$H^{k-r}(\M_{g,p}^b(\ell);\fH_{g,p}^b(\ell;\Q)^{\otimes r})$$ in the following way. Recall that $\fH_{g,p}^b(\ell;\Q) = H^1(S_{\mathcal{D}};\Q)$, where $S_{\mathcal{D}} \to \Sigma_{g,p}^b$ is the regular cover with deck group $\mathcal{D}=H_1(\Sigma_g;\mathbb{Z}/\ell)$. The covering map induces a natural map between tensor powers:
$$H^1(\Sigma_{g,p}^b;\Q)^{\otimes r} \to \fH_{g,p}^b(\ell;\Q)^{\otimes r},$$
which in turn induces a map on twisted cohomology:
\begin{equation}\label{map1}
H^{k-r}(\M_{g,p}^b;H^1(\Sigma_{g,p}^b;\Q)^{\otimes r}) \to H^{k-r}(\M_{g,p}^b(\ell);\fH_{g,p}^b(\ell;\Q)^{\otimes r}).
\end{equation}

Putman studied the case $r = 1$ and showed that the map \eqref{map1} is an isomorphism when $g \ge 2k^2 + 7k + 1$ (\cite[Theorem C]{AndyStable}). This result implies that for $r = 1$, the right-hand side of \eqref{map1} is independent of $g$, $b$, and $\ell$, since the left-hand side is. 

For general $r \ge 2$, Putman conjectured the following:
\begin{conj}[Putman \protect{\cite[Remark 1.8]{AndyStable}}]
For $r\ge 2$, the map \eqref{map1} is not an isomorphism.
\end{conj}  
Since Theorem \ref{main2} shows that the right-hand side of \eqref{map1} is not stable with respect to the genus $g$, in contrast with the known stability of the left-hand side in \eqref{hmh}, we show that Putman's conjecture is true:

\begin{cor}
For $r\ge 2$ and $k$ even, the map \eqref{map1} is not an isomorphism when $g\ge \text{max}(\frac{3}{2} k+1,2k^2+7k+2)=2k^2+7k+2$. 
\end{cor}
\begin{remark}
For $k$ odd, both sides of the map \eqref{map1} vanish when $g\ge 2k^2+7k+2$. 
\end{remark}

\subsection{An Example} To gain more insight into how these twisted cohomology groups in the map \eqref{map1} differ when $r\ge 2$, we present an example for the case $r=2$. By taking $r=2$ in Theorem \ref{thmA}, we obtain
\[ H^{\bullet-2}(\M_{g,p}^b;H^1(\Sigma_{g,p}^b;\Q)^{\otimes 2})
\cong H^{\bullet}(\M_{g,p}^b;\Q)\otimes \left( 
\begin{array}{c}
u_1 u_2\,\Q[u_1,u_2] \\[1mm]
\oplus\, \Q[u_{\{1,2\}}]\,a_{\{1,2\}}
\end{array}
\right), \]
in degrees $k$ such that $g\ge \frac{3}{2} k+1$, with $deg(u_1)=deg(u_2)=deg(a_{\{1,2\}})=2$. Notably, this twisted cohomology is independent of the genus $g$.

In contrast, taking $r=2$ in Theorem \ref{main2} yields
\[
H^{\bullet-2}(\M_{g,p}^b(\ell);\fH_{g,p}^b(\ell;\Q)^{\otimes 2})
\cong H^{\bullet}(\M_{g,p}^b(\ell);\Q)\otimes \left(
\begin{array}{c}
v_1 v_2\,\Q[v_1,v_2] \\[1mm]
\oplus\, \displaystyle \bigoplus_{d\in \mathcal{D}} \Q[v_{(\{1<2\},d)}]\,a_{(\{1<2\},d)}
\end{array}
\right),
\]
in degrees $k$ such that $g\ge 2k^2+7k+2$,
with $\deg(v_1)=\deg(v_2)=\deg(a_{(\{1<2\},d)})=2$. Since there is a nontrivial summand corresponding to each 
$
d\in \mathcal{D}=(\mathbb{Z}/\ell)^{2g},
$
the twisted cohomology group $H^{\bullet-2}(\M_{g,p}^b(\ell);\fH_{g,p}^b(\ell;\Q)^{\otimes 2})$ exhibits a dependence on the genus $g$.

The key distinction between these two twisted cohomology groups lies in the nature of the cohomology classes $a_{\{1,2\}}$ and $a_{(\{1<2\},d)}$. The class $a_{\{1,2\}}$ is associated with the moduli space $\mathcal{C}_{g,2}$ of closed genus-$g$ Riemann surfaces with two (not necessarily distinct) marked points. More precisely, $a_{\{1,2\}}$ represents the Poincar\'e dual of the subvariety of $\mathcal{C}_{g,2}$ where the two marked points coincide (i.e. $x_2=x_1$).

In contrast, for each $d\in \mathcal{D}=(\mathbb{Z}/\ell)^{2g}$, the class $a_{(\{1<2\},d)}$ relates to the moduli space $\mathcal{C}_{g,2}(\ell)$ that we construct in Section 4. This moduli space is a bordification of $\mathcal{M}_{g,2}(\ell)$, the moduli space of Riemann surfaces homeomorphic to $\Sigma_{g,2}$ with level-$\ell$ structures. One can roughly view $\mathcal{C}_{g,2}(\ell)$ as the moduli space of two (not necessarily distinct) marked points on the regular $\mathcal{D}$-cover of a closed genus-$g$ Riemann surface. The class $a_{\{1<2\},d}$ is defined as the Poincar\'e dual of the subvariety of $\mathcal{C}_{g,2}(\ell)$ where the marked points $y_1,y_2$ on the $\mathcal{D}$-cover satisfy the relation $y_2=d\cdot y_1$. 

This distinction illustrates that the additional classes in $H^{\bullet-2}(\M_{g,p}^b(\ell);\fH_{g,p}^b(\ell;\Q)^{\otimes 2})$  arise from the richer geometric structure of the moduli space $\mathcal{C}_{g,2}(\ell)$ compared to $\mathcal{C}_{g,2}$.

\subsection{Proof ideas}
To conclude the introduction, we outline our approach to proving Theorem \ref{main2}.  

Our strategy is inspired by Looijenga’s computation of $H^{\bullet}(\M_g; H^1(\Sigma_g; \Q)^{\otimes r})$ (\cite[Corollary 3.3]{LooiStable}). Looijenga studied the Leray spectral sequence associated with the $r$-fold fiber product of the universal curve $\mathcal{M}_{g,1} \to \mathcal{M}_g$. Since we work with the level-$\ell$ mapping class group, we shift our focus from $\mathcal{M}_g$ to $\mathcal{M}_{g,r}(\ell)$, the moduli space of Riemann surfaces homeomorphic to $\Sigma_{g,r}$ equipped with a level-$\ell$ structure. Our approach proceeds as follows:

\begin{enumerate}
    \item We reinterpret $\mathcal{M}_{g,r}(\ell)$ as the moduli space of one marked point on $\Sigma_{g,r}$ and $(r-1)$ marked points on the regular $H_1(\Sigma_g; \mathbb{Z}/\ell)$-cover of $\Sigma_{g,r}$, subject to certain conditions.
    
    \item We construct a partial bordification of $\mathcal{M}_{g,r}(\ell)$, denoted $\mathcal{C}_{g,r}(\ell)$, whose boundary strata are described in terms of orbit configuration spaces. We remark that, unlike Looijenga's bordification, this space $\mathcal{C}_{g,r}(\ell)$ is not the total space of the $r$-fold fiber product of $\mathcal{M}_{g,1}(\ell) \to \mathcal{M}_{g}(\ell)$.
    
    \item We compute the rational cohomology of $\mathcal{C}_{g,r}(\ell)$ for $g$ sufficiently large, stated in Theorem \ref{coh2}.
    
    \item We construct a map $\mathcal{C}_{g,r+1}(\ell) \to \mathcal{M}_{g,1}(\ell)$, where the cohomology of the fiber contains a summand isomorphic to $\fH_{g,1}(\ell;\Q)^{\otimes r}$, the $r$-fold tensor power of the Prym representation.
    
    \item By a theorem of Deligne (\cite[Sec.~8.1]{deligne}, see Theorem \ref{ever}), the Leray spectral sequence associated with the map $\mathcal{C}_{g,r+1}(\ell) \to \mathcal{M}_{g,1}(\ell)$ degenerates at the $E_2$-page.
    
    \item We compute the $E_2$-page of the Leray spectral sequence, and using mixed Hodge theory, we identify $H^{\bullet}(\M_{g,1}(\ell); \fH_{g,1}(\ell; \Q)^{\otimes r})$ as a subspace of $H^{\bullet}(\mathcal{C}_{g,r+1}(\ell); \Q)$.
    
    \item Finally, to extend our results from $\M_{g,1}(\ell)$ to $\M_{g,p}^b(\ell)$ with $p + b \ge 1$, we apply Putman's theory of partial level-$\ell$ representations (\cite{AndyStable}).
\end{enumerate}

\noindent \textbf{Outline.} In Section 2, we introduce key preliminaries, including background on mapping class groups, group cohomology, and mixed Hodge theory. In Section 3, we synthesize the results of Looijenga and Kawazumi to give a complete account of the cohomology groups \eqref{hmh}. In Section 4, we construct the space $\mathcal{C}_{g,r}(\ell)$ and compute its rational cohomology for $g$ sufficiently large. In Section 5, we present the proof of Theorem \ref{main2}. In Section 6, we prove Theorem \ref{main1}.

\noindent \textbf{Acknowledgments.} I sincerely thank my advisor Andrew Putman for suggesting the problem, providing many helpful comments, and offering continuous encouragement throughout the entire project, as well as for numerous suggestions on earlier drafts of the paper. I would also like to thank Eduard Looijenga, Oscar Randal-Williams, and Eric Riedl for helpful conversations. Finally, I thank the anonymous referees for their many useful comments.

\section{Preliminaries}

\subsection{Stable cohomology of the mapping class group}

We summarize key stability results for the cohomology of the mapping class group.

\begin{thm}[Harer stability \cite{harer}]
$H^{k}(\M_{g,p}^b;\mathbb{Z})$ is independent of $g$ and $b$ in degrees $\le N(g)$.
\end{thm}

\begin{remark} Here the number $N(g)$ is the maximal degree $N$ such that the two homomorphisms $$H^N(\M_{g+1,p}^b;\mathbb{Z})\to H^N(\M_{g,p}^{b+1};\mathbb{Z})$$ and $$H^N(\M_{g,p}^{b};\mathbb{Z})\to H^N(\M_{g,p}^{b+1};\mathbb{Z})$$ are isomorphisms. The original bound given by Harer (\cite{harer}) was $N(g;\mathbb{Z})\ge\frac{1}{3} g$, which was later improved by Ivanov (\cite{ivanov}), Boldsen (\cite{boldsen}) and Randal-Williams (\cite{rw}) to $N(g)\ge \frac{2}{3}(g-1)$.
\end{remark}

Unlike the genus $g$ and boundary components $b$, the stable cohomology of $\M_{g,p}^b$ does depend on the number of punctures $p$. For $1\le i \le p$, we have an Euler class $e_i\in H^2(\Sigma_{g,p}^b;\mathbb{Z})$, as we mentioned in the introduction. Below, we provide an alternate definition following Looijenga (\cite[Sec.~2]{LooiStable}) for future purposes:

\begin{itemize}
\item Let $\mathcal{M}_{g,p}$ denote the moduli space of closed Riemann surfaces of genus $g$ with $p$ distinct marked points. Since $\mathcal{M}_{g,p}$ is a rational classifying space for the mapping class group, we have $$H^\bullet(\mathcal{M}_{g,p};\Q)\cong H^\bullet(\M_{g,p};\Q).$$
\item Consider the universal curve $\pi: \mathcal{M}_{g,1}\to \mathcal{M}_g$. Let $\theta$ be its relative tangent sheaf.
\item For each $i$, define $f_i:\mathcal{M}_{g,p}\to \mathcal{M}_{g,1}$ as the map that forgets all but the $i$-th marked point. Then, we define the Euler class $e_i$ as the first Chern class: 
$$e_i=c_1(f_i^*(\theta))\in H^2(\M_{g,p};\Q).$$
\end{itemize} 

\begin{thm}[Looijenga \protect{\cite[Prop.~2.2]{LooiStable}}]\label{wuwu}
The ring homomorphism 
$$H^\bullet(\M_g;\mathbb{Z})[e_1,\cdots,e_p]\to H^{\bullet}(\M_{g,p}^b;\mathbb{Z})$$ is an isomorphism in degrees $\le N(g)$.
\end{thm}

The stable integral cohomology of the mapping class group is complicated, but the stable rational cohomology has a beautiful form. Conjectured by Mumford (\cite{mumford}), proved by Madsen and Weiss (\cite[Theorem 1.1.1]{Madsen}), the stable rational cohomology of $\M_g$ is isomorphic to polynomial ring generated by these Miller-Morita-Mumford classes (\cite{miller}, \cite{morita}) $\kappa_i\in H^{2i}(\M_g;\Q)$ for $i\ge 1$. See also \cite{hatcher}, \cite{galatius}, and \cite{wahl} for alternate proofs and expositions.

\begin{thm}[Madsen-Weiss \protect{\cite[Theorem 1.1.1]{Madsen}}]\label{2.2} In degrees $\le N(g)$, we have
$$H^{\bullet}(\M_g;\mathbb{Q})\cong \mathbb{Q}[\kappa_1,\kappa_2,\kappa_3,\cdots].$$
\end{thm}

Combining these two theorems above, in degrees $\le N(g)$, we have
$$H^{\bullet}(\M_{g,p}^b;\mathbb{Q})\cong \mathbb{Q}[e_1,\cdots,e_p]\otimes\Q[\kappa_1,\kappa_2,\kappa_3,\cdots].$$
 
\subsection{Level-$l$ mapping class groups}
Recall that the level-$\ell$ mapping class group $\M_{g,p}^b(\ell)$ is the subgroup of $\M_{g,p}^b$ which acts trivially on $H_1(\Sigma_{g,p}^b;\mathbb{Z}/\ell)$. It has many similar properties to $\M_{g,p}^b$. For example:
\begin{prop}[\protect{\cite[Prop.~2.10]{AndyStable}}]\label{lBirman}
Fix some $g, p, b \ge 0$ such that $\pi_1(\Sigma_{g,p}^{b+1})$ is nonabelian, and let $\partial$ be a
boundary component of $\Sigma_{g,p}^{b+1}$. Let $\ell \ge 2$. Then there is a central extension
$$1 \to \mathbb{Z} \to \M^{b+1}_{g,p}(\ell) \to\M^b_{g,p+1}(\ell)\to 1,$$
where the central $\mathbb{Z}$ is generated by the Dehn twist $T_{\partial}$.
\end{prop}

We also have the mod-$\ell$ version of the Birman exact sequence:

\begin{thm}[Mod-$l$ Birman exact sequence, Putman \protect{\cite[Theorem 2.8]{AndyStable}}]\label{Birman}
Fix $g,p,b\ge 0,\ell\ge 2$ such that $\pi_1(\Sigma_{g,p}^b)$ is non-abelian. Let $x_0$ be a puncture of $\Sigma_{g,p+1}^b$. There is a short exact sequence obtained by forgetting $x_0$:
$$1\to PP_{x_0}(\Sigma_{g,p}^b;\ell)\to \M_{g,p+1}^b(\ell)\to \M_{g,p}^b(\ell)\to 1,$$
where the level-$\ell$ point pushing group $PP_{x_0}(\Sigma_{g,p}^b;\ell)$ is as follows:
\begin{itemize}
\item If $p=b=0$, then $PP_{x_0}(\Sigma_{g,p}^b;\ell)=\pi_1(\Sigma_{g,p}^b,x_0)$.
\item If $p+b\ge 1$, then $PP_{x_0}(\Sigma_{g,p}^b;\ell)=\textup{Ker}(\pi_1(\Sigma_{g,p}^b,x_0)\to H_1(\Sigma_{g,p}^b;\mathbb{Z}/\ell)\to H_1(\Sigma_g;\mathbb{Z}/\ell))$.
\end{itemize}
\end{thm}

A natural question is whether the finite-index subgroup $\M_{g,p}^b(\ell)$ of $\M_{g,p}^b$ has the same stable cohomology as $\M_{g,p}^b$. The answer is negative for integral cohomology, as exotic torsion elements in $H^1(\M_{g,p}^b(\ell);\mathbb{Z})$ have been discovered by Perron \cite{perron}, Sato \cite{sato}, and Putman \cite{AndyPicard}. These elements do not arise from $H^1(\M_{g,p}^b;\mathbb{Z})$, and they are not stable. 

However, Putman established that the result does hold for rational cohomology: 
\begin{thm}[Putman, \protect{\cite[Theorem A]{AndyStable}}]\label{andy0}
We have
$$H^k(\M_{g,p}^b(\ell);\Q) \cong H^k(\M_{g,p}^b; \Q) \text{ if } g\ge 2k^2 + 7k + 2.$$ 
\end{thm}

Putman also studied the twisted cohomology of $\M_{g,p}^b(\ell)$ with coefficients in tensor powers of $H^1(\Sigma_{g,p}^b;\Q)$ and showed this cohomology is stable by the following theorem:
\begin{thm}[Putman \protect{\cite[Theorem B]{AndyStable}}] The following map is an isomorphism:
$$H^k(\M_{g,p}^b;H^1(\Sigma_{g,p}^b;\Q)^{\otimes r})\to H^k(\M_{g,p}^b(\ell);H^1(\Sigma_{g,p}^b;\Q)^{\otimes r}),$$
when $g \ge 2(k + r)^2 + 7k + 6r + 2$.
\end{thm}

\subsection{Group cohomology} Here we review some useful theorems about group cohomology. We will apply these theorems in later sections.

Consider a group $G$ and a finite-index subgroup $H$, along with a $\mathbb{Z}[G]$-module $M$. There is a natural map $\textup{Res}_H^G: H^k(G;M)\to H^k(H;M)$ in group cohomology obtained from the inclusion $\mathbb{Z}[H]\hookrightarrow\mathbb{Z}[G]$. We call it the restriction map. There is a "wrong-way" map, called the transfer map, $\text{cor}_H^G:H^k(H;M)\to H^k(G;M)$ which satisfies:

\begin{prop}[\protect{\cite[Prop.~9.5]{brown}}]\label{transfer}
If $H$ is a finite index subgroup of $G$ with index $[G:H]$, then the composition of transfer maps and restriction maps is the multiplication map by $[G:H]$, i.e. $\text{cor}_H^G\cdot \textup{Res}_H^G=[G:H]\textup{id}$.
\end{prop}

\begin{remark} In particular, supposing $H$ is a finite-index subgroup of $G$, if $M$ is a $\Q$(or $\mathbb{R}$)-vector space, we see that $\text{cor}_H^G$ is surjective and $\textup{Res}_H^G$ is injective.
\end{remark}

\begin{prop}[\protect{\cite[Prop.~1.1]{note}}]\label{quo}
Let $X$ be a simplicial complex and $G$ be a finite group acting
simplicially on X. Then for all $k \ge 0$, we have
$$H^k(X/G;\Q) \cong (H^k(X;\Q))^G.$$
\end{prop}

The next technique we will use a lot is the Gysin sequence, which is a long exact sequence obtained from the second page of  the Hochschild-Serre spectral sequence (\cite[prop.~7]{serre}) associated to a central extension of groups:

\begin{prop}[Gysin Sequence \protect{\cite[prop.~7]{serre}}]\label{gy}
Consider a central extension
$$1\to \mathbb{Z}\to G\to K\to 1$$
and a $\mathbb{Z}[K]$ module $M$ (thus $M$ is also a  $\mathbb{Z}[G]$ module through the map $G\to K$). We have the following long exact sequence:
$$\begin{aligned} \cdots \to H^{k-2}(K;M)\to H^k(K;M)\to H^k(G;M)\to H^{k-1}(K;M)\to H^{k+1}(K;M)\to \cdots \end{aligned},$$
where $H^{k-2}(K;M)\to H^k(K;M)$ is the differential on the $E_2$-page of the Hochschild-Serre spectral sequence.
\end{prop}

\begin{remark} The geometric version of the Gysin sequence is that, for an oriented sphere bundle $S^d \hookrightarrow E \to M$, we have the following long exact sequence
$$ \cdots \to H^{k-d-1}(M)\to H^{k}(M) \to H^k (E) \to H^{k-d}(M)\to H^{k+1}(M)\to \cdots,$$
where the map $H^{k-d-1}(M)\to H^{k}(M)$ is the wedge product with the Euler class, and the map $H^k (E) \to H^{k-d}(M)$ is fiberwise integration.
\end{remark}

There is a generalization of the Gysin sequence, called the Thom-Gysin Sequence. This sequence is widely known, appearing in \cite[eq.~1.16]{yang} and \cite[p.~167]{MFK} for example.

\begin{prop}[Thom-Gysin Sequence \protect{\cite[eq.~1.16]{yang}}]\label{tg}
Let $X$ be a smooth complex variety, and let $Y$ be an open smooth subvariety of $X$ whose complement $X\setminus Y$ has constant (real) codimension $d$. We then have the following long exact sequence:
$$\cdots \to H^{k-d}(X\setminus Y;\Q)\to H^k(X;\Q)\to H^k(Y;\Q)\to H^{k-d+1}(X\setminus Y;\Q) \to H^{k+1}(X;\Q)\to \cdots.$$
\end{prop}

\begin{remark}
When $X$ is the total space of a sphere bundle and $X\setminus Y$ is the zero section, the associated Thom-Gysin sequence is exactly the Gysin sequence. In general, one can derive the Thom-Gysin sequence by applying the Thom Isomorphism Theorem and the Excision Theorem to the tubular neighborhood of $X\setminus Y$ in $X$.
\end{remark}

\begin{remark}
In most of our cases in Section 4, we apply the Thom-Gysin sequence to quasi-projective varieties, such as $\mathcal{M}_{g,r}(\ell)$ for $\ell \ge 3$. The case of $\mathcal{M}_{g,r}(2)$ is the only instance where we use the Thom-Gysin sequence on a quasi-projective orbifold, which is a quotient of a quasi-projective variety by a finite group. In this case, by Proposition \ref{quo}, the Thom-Gysin sequence retains the same form, so we proceed without additional markings.
\end{remark}

\subsection{Deligne's degeneration theorem} Our computations will be based on an important theorem due to Deligne (\cite[Sec.~8.1]{deligne}), which was proved using mixed Hodge theory. Here we quote the version in Griffiths and Schimid's survey \cite[p.~42]{mixed}:
\begin{thm}[Deligne's degeneration theorem \protect{\cite[Sec.~8.1]{deligne}}, \protect{\cite[p.~42]{mixed}}]\label{ever}
Let $E$ be a K\"ahler manifold, $X$ a complex manifold, and $f:E\to X$ a smooth, proper holomorphic mapping, which implies $f$ is a differential fiber bundle whose fibers $X_b,b\in B$ are compact K\"ahler manifolds. The corresponding Leray spectral sequence 
$$E_2^{p,q}=H^p(B,R^q f_*(\Q))\Rightarrow H^{p+q}(E;\Q),$$ 
degenerates at the second page, i.e. $E_2=E_g$. Here $$R^q f_*(\Q) \text{ comes from the presheaf }  U\mapsto H^q(f^{-1}(U);\Q).$$
\end{thm}

\begin{remark} Smooth quasi-projective varieties are K\"ahler manifolds, so we can apply Deligne's degeneration theorem to $\mathcal{C}_{g,r}(\ell)\to \mathcal{M}_{g,1}(\ell)$ for $\ell\ge 3$. This covers nearly all cases in Section 5, though there are two cases where we apply Deligne's degeneration theorem to quasi-projective orbifolds: for $\mathcal{C}_{g,r}\to \mathcal{M}_{g,1}$ in Section 3 and for $\mathcal{C}_{g,r}(2)\to \mathcal{M}_{g,1}(2)$ in Section 5. Deligne's degeneration theorem still applies in these cases, as demonstrated by the following argument:

Let $E$ be a quasi-projective orbifold and $X$ a complex orbifold. Suppose $f:E\to X$ is a map of orbifolds that lifts to a map $\widetilde{f}:E'\to X'$ with the following properties:
\begin{itemize}
\item $E'$ and $B'$ are quasi-projective varieties.
\item $E'\to E$ and $B'\to B$ are finite branched $G$-covers for a finite group $G$.
\item $\widetilde{f}:E'\to X'$ is a smooth, proper holomorphic mapping.
\end{itemize}
By applying Deligne's degeneration theorem \ref{ever} to $\widetilde{f}:E'\to X'$, we obtain the Leray spectral sequence
\begin{equation}\label{e2}
E_2^{p,q}=H^p(B',R^q \widetilde{f}_*(\Q))\Rightarrow H^{p+q}(E';\Q),\end{equation}
which degenerates at the second page.

By Proposition \ref{quo}, we have
$$H^{p+q}(E;\Q)\cong H^{p+q}(E';\Q)^G,$$
and by a similar argument, we obtain
$$H^p(B,R^q f_*(\Q))\defeq H^p(B',R^q \widetilde{f}_*(\Q))^G.$$
Restricting the spectral sequence (\ref{e2}) to $G$-invariants, we conclude that the Leray spectral sequence
$$E_2^{p,q}=H^p(B,R^q f_*(\Q))\Rightarrow H^{p+q}(E;\Q)$$ 
also degenerates at the second page.
\end{remark}

\subsection{Mixed Hodge theory} Mixed Hodge theory is used in the proof of Deligne's degeneration theorem \ref{ever}, and is also a powerful tool for determining terms in spectral sequences. We will introduce some basic properties according to the survey \cite{mixed} by Griffiths and Schimid. First, we start with definitions of pure Hodge structures.

\begin{defn}[\protect{\cite[Definition 1.1, 1.2]{mixed}}]
Let $H_{\mathbb{R}}$ be a finite dimensional real vector space, and $H_{\mathbb{Z}}$ be a lattice in $H_{\mathbb{R}}$. Let $H=H_{\mathbb{R}}\otimes_{\mathbb{R}}\mathbb{C}$ be its complexification.
\begin{enumerate}
\item A \textbf{Hodge structure of weight $m$} on $H$ consists of a direct sum decomposition
$$H=\bigoplus_{p+q=m} H^{p,q}, \text{ with } H^{q,p}=\bar{H}^{p,q},$$
where $\bar{H}^{p,q}$ denotes the complex conjugate of $H^{p,q}$.
\item A \textbf{morphism of Hodge structures of type $(r,r)$} is a linear map (defined over $\Q$ relative to the lattices $H_{\mathbb{Z}}$, $H'_{\mathbb{Z}}$) 
$$\varphi: H\to H', \text{ with } \varphi(H^{p,q})\subset (H')^{p+r,q+r}.$$
\item A Hodge structure $H$ of weight $m$ is \textbf{polarized} by a non-degenerate integer bilinear form $Q$ on $H_\mathbb{Z}$ if   the extended bilinear form $Q$ on $H$ satisfies the following conditions
\begin{align*}Q(v,w)=(-1)^m Q(w,v), \forall  v,w\in H,\\
Q(H^{p,q},H^{p',q'})=0, \text{ unless } p=q', q=p',\\
\sqrt{-1}^{p-q} Q(v,\bar{v}) >0, \text{ for } v\in H^{p,q}, v\neq 0.
\end{align*}
\end{enumerate}
\end{defn}

\begin{remark} Let $H$ be a Hodge structure of weight $m$ and $H'$ be a Hodge structure of weight $m'$. The tensor product $H\otimes H'$ inherits a Hodge structure of weight $m+m'$: $$H\otimes H'=\sum\limits_{p+q=m+m'} H^{p,m-p}\otimes (H')^{q,m'-q}.$$ Moreover, if $H$ is polarized by $Q$ and $H'$ is polarized by $Q'$, then $H\otimes H'$ is polarized by the induced bilinear form $Q\otimes Q'$.
\end{remark}

A Hodge structure can also be given as follows:

\begin{prop}[\protect{\cite[p.~35]{mixed}}]
Let $H, H_{\mathbb{R}}, H_{\mathbb{Z}}$ be the same as above.
\begin{enumerate}
\item There is a Hodge structure of weight $m$ on $H$ if and only if $H$ has a \textbf{Hodge filtration}
$$ H\supset\cdots \supset F^{p-1}\supset F^{p} \supset F^{p+1}\supset \cdots \supset 0, $$
$$ \text{with } F^p\oplus \bar{F}^{m-p+1} \xrightarrow{\cong}H, \text{ for all }p.$$
\item A map $\varphi: H\to H'$ is a morphism of Hodge structures of type $(r,r)$ if and only if $\varphi$ preserves the Hodge filtration with a shift by $r$, i.e.
$$\varphi(F^p)\subset (F')^{p+r}, \text{ for all } p.$$
In particular, a morphism of Hodge structures of type $(r,r)$ preserves the Hodge filtration strictly:
$$\varphi(F^p)=(F')^{p+r}\cap \text{Im}(\varphi), \text{ for all }p.$$ 
\end{enumerate}
\end{prop}

A mixed Hodge structure is a generalization of a Hodge structure.
\begin{defn}[\protect{\cite[Definition 1.11]{mixed}}]
Let $H_{\mathbb{Z}}$ be a finitely generated free abelian group.
\begin{enumerate}
\item A \textbf{mixed Hodge structure} is a triple $(H_{\mathbb{Z}},W_\bullet,F^\bullet)$ such that
\begin{enumerate}
\item The weight filtration $W_\bullet$ is 
$$0\subset \cdots \subset W_{m-1} \subset W_m \subset W_{m+1}\subset \cdots \subset H_{\mathbb{Z}}\otimes_{\mathbb{Z}} \Q=H_{\Q}.$$
\item The Hodge filtration $F^\bullet$ is
$$ H=H_{\mathbb{Z}}\otimes_{\mathbb{Z}}\mathbb{C} \supset \cdots \supset F^{p-1}\supset F^p\supset F^{p+1} \supset \cdots \supset 0.$$
\item For each $m\in \mathbb{Z}$, on the graded piece $\text{Gr}_m(W_\bullet)=W_{m}/W_{m-1}$, the induced filtration by $F^\bullet$ defines a Hodge structure of weight $m$.
\end{enumerate}
\item A \textbf{morphism of mixed Hodge structures} of type $(r,r)$ consists of a linear map
$$\varphi: H_{\Q}\to (H')_{\Q} \text{ with } \varphi(W_m)\subset (W')_{m+2r}, \text{ and }\varphi(F^p)\subset (F')^{p+r}.$$
\end{enumerate}
\end{defn}

The morphisms of mixed Hodge structures are also strict in the following sense.

\begin{lem}[\protect{\cite[Lemma 1.13]{mixed}}]
A morphism of type $(r,r)$ between mixed Hodge structures is strict with respect to both the weight and Hodge filtrations. More precisely, 
$$\varphi(W_m)=(W')_{m+2r}\cap \text{Im}(\varphi), \ \varphi(F^p)=(F')^{p+r}\cap \text{Im}(\varphi).$$
\end{lem}

\begin{remark}
Let $(H_{\mathbb{Z}},W_\bullet,F^\bullet)$ and $(\tilde{H}_{\mathbb{Z}},\tilde{W}_\bullet,\tilde{F}^\bullet)$ be two mixed Hodge structures. Their tensor product $H\otimes \tilde{H}$ inherits a mixed Hodge structure with the weight filtration
$$0 \subset \cdots \subset \sum\limits_{a+b\le m-1} W_a\otimes \tilde{W}_b  \subset \sum\limits_{a+b\le m} W_a\otimes \tilde{W}_b   \subset \sum\limits_{a+b\le m+1} W_a\otimes \tilde{W}_b \subset \cdots \subset  H_{\Q}\otimes \tilde{H}_{\Q},$$
and the Hodge filtration
$$ H\otimes \tilde{H} \supset \cdots \supset \sum\limits_{a+b\ge p-1} F^a\otimes \tilde{F}^b \supset \sum\limits_{a+b\ge p} F^a\otimes \tilde{F}^b \supset \sum\limits_{a+b\ge p+1} F^a\otimes \tilde{F}^b \supset \cdots \supset 0.$$
\end{remark}

We are interested in the cohomology of quasi-projective complex varieties, which has a canonical polarizable mixed Hodge structure by the following theorem of Deligne:
\begin{thm}[Deligne \protect{\cite[Theorem 3.2.5]{deII}}]\label{polar}
Let $X$ be a quasi-projective complex variety. Then $H^\bullet(X;\Q)$ carries a canonical polarizable mixed Hodge structure.
\end{thm}

Here a polarizable mixed Hodge structure means all graded pieces $\text{Gr}_m(W_\bullet)$ are polarizable Hodge structures. We can decompose a polarized Hodge structure into a direct sum of simple objects by the following theorem:

\begin{thm}[\protect{\cite[Corollary 2.12]{peters}}]\label{semi}
The category of polarizable Hodge structures of weight $m$ is semi-simple.
\end{thm}

\section{Stable Cohomology of $\textup{Mod}(\Sigma_{g,p}^b)$ with Coefficients in $H^1(\Sigma_{g,p}^b;\mathbb{Q})^{\otimes r}$}

In this section, we will review Looijenga's results \cite{LooiStable} and Kawazumi's results \cite{kawa} and give an explicit expression of $H^{\bullet}(\M_{g,p}^b;H^1(\Sigma_{g,p}^b;\mathbb{Q})^{\otimes r})$ for $p+b\ge 1$.

\subsection{Looijenga's results}
Let $\mathcal{C}_{g,r}$ be the moduli space of closed genus-$g$ Riemann surfaces with $r$ (not necessarily distinct) ordered marked points. The map $\mathcal{C}_{g,r}\to \mathcal{M}_g$ obtained by forgetting all the marked points can also be viewed as the $r$-fold fiber product of the universal curve $\pi: \mathcal{M}_{g,1}\to\mathcal{M}_g$.

 Looijenga computed the rational stable cohomology of $\mathcal{C}_{g,r}$ in a stable range, in terms of the following cohomology classes:
\begin{enumerate}
\item For $1\le i\le r$, let $$u_i\in H^2(\mathcal{C}_{g,r};\Q)$$ be the first Chern class of $\theta_i=f_i^*(\theta)$, where $f_i:\mathcal{C}_{g,r}\to \mathcal{M}_{g,1}$ is the map forgetting all but the $i$-th marked point, and $\theta$ is the relative tangent sheaf of $\pi:\mathcal{M}_{g,1}\to \mathcal{M}_g$.
\item For a subset $I$ of $[r]=\{1,2,\cdots,r\}$ with $|I|\ge2$, let $$a_I\in H^{2|I|-2}(\mathcal{C}_{g,r};\Q)$$ be the Poincar\'e dual of the subvariety of $\mathcal{C}_{g,r}$ where the marked points $\{x_i\}_{i\in I}$ coincide:
$$x_i=x_j, \text{ for all }i,j\in I.$$
\end{enumerate}
Looijenga (\cite[Lemma 2.4]{LooiStable}) showed that these cohomology classes satisfy the relations
\begin{equation}\label{ua1} 
 u_i a_I=u_j a_I, \text{ if } i,j\in I;
 \end{equation}
 \begin{equation}\label{ua2}
 a_I a_J=u_i^{|I\cap J|-1} a_{I\cup J}, \text{ if } i\in I\cap J \neq \emptyset.
\end{equation}
We now state Looijenga's theorem on the rational cohomology of $\mathcal{C}_{g,r}$.
\begin{thm}[Looijenga \protect{\cite[Theorem 2.3]{LooiStable}}]\label{wowLooi}
Let $A^{\bullet}_r$ denote the graded commutative $\Q$-algebra generated by all $u_i$ and $a_I$ subject to the relations (\ref{ua1}) (\ref{ua2}).
There is an algebra homomorphism
$$ H^\bullet(\M_{g};\Q) \otimes A^{\bullet}_r \to H^\bullet (\mathcal{C}_{g,r};\mathbb{Q})$$ which is an isomorphism in degrees $k$ such that $g\ge \frac{3}{2}k+1$.
\end{thm}

\begin{remark}\label{rmk1} Looijenga provided an explicit description of $A^{\bullet}_r$ as a vector space:
$$A^{\bullet}_r=\bigoplus_{P|[r]} \mathbb{Q}[u_I:I\in P]a_P, \text{ where } a_P=\prod\limits_{I\in P,|I|\ge 2} a_I.$$ 
Here, the notation is as follows:
\begin{itemize}
\item $P|[r]$ denotes a \textbf{partition} of the set $[r]$, i.e. $P=\{ I_1,I_2,\cdots, I_m\}$ where $I_i$ are disjoint nonempty subsets of $[r]$, and $I_1\cup I_2 \cup\cdots\cup I_m=[r]$. The class $a_P=\prod\limits_{I\in P,|I|\ge 2} a_I$ represents the Poincar\'e dual of the subvariety of $\mathcal{C}_{g,r}$ where the marked points indexed by $I_a\in P$ coincide for each $1\le a\le m$.
\item For a given partition $P$ of $[r]$, we define an equivalence relation on $\{u_1,\cdots,u_r\}$:
$$u_i \sim u_j \text{ if and only if }i,j\in I \text{ for some }I\in P.$$ The equivalence class of $u_i$ for $i\in I$ is denoted by $u_I$, with the convention that if $I$ is a singleton $\{i\}$, then $u_I=u_i$.
\end{itemize}
By observing the two relations (\ref{ua1}) (\ref{ua2}) that $u_i,a_I$ satisfy, it is not hard to derive the expression of $A^{\bullet}_r$ above. There is a simple mixed Hodge structure on $A^{\bullet}_r$ where the degree-$2m$ part has Hodge type $(m,m)$. Since $H^\bullet(\M_{g};\Q)$ carries a canonical mixed Hodge structure (see e.g. \cite{LooiStable}), this induces a mixed Hodge structure on $H^\bullet (\mathcal{C}_{g,r};\mathbb{Q})$.

\end{remark}

To study the twisted cohomology of $\M_g$, Looijenga introduced a key subspace of $A^{\bullet}_r$ (\cite[Corollary 2.7]{LooiStable}):

\begin{equation}\label{a'} A'^{\bullet}_r= \bigoplus_{P|[r]} (\prod_{\{i\}\in P} u_i^2)\mathbb{Q}[u_I:I\in P] a_P.\end{equation}
By analyzing the Leray spectral sequence associated with the forgetful map $\mathcal{C}_{g,r}\to \mathcal{M}_g$, Looijenga embedded the twisted cohomology groups of $\M_g$ with coefficients $H^1(\Sigma_g;\Q)^{\otimes r}$ into $H^{\bullet}(\mathcal{C}_{g,r};\Q)$, leading to the following theorem:

\begin{thm}[Looijenga \protect{\cite[Corollary 3.3]{LooiStable}}]\label{lolo}
Let $A'^{\bullet}_r$ be as defined in \eqref{a'}. There is a graded $S_r$-equivariant map of $H^\bullet(\M_{g};\Q)$-modules, which is also a morphism of mixed Hodge structures:
$$ H^\bullet(\M_{g};\Q)\otimes A'^{\bullet}_r\to H^{\bullet-r}(\textup{Mod}_g;H^1(\Sigma_g;\mathbb{Q})^{\otimes r}),$$
which is an isomorphism in degrees $k$ such that $g\ge \frac{3}{2}k+1$.
\end{thm}

Here, the symmetric group $S_r$ acts on the left-hand side by permuting the indices in $[r]$, and on the right-hand side by permuting the tensor factors of $H^1(\Sigma_g;\mathbb{Q})^{\otimes r}$. This equivariance allows Looijenga to compute the cohomology of $\M_g$ with coefficients in any Schur functors of $Sp_{2g}(\mathbb{C})$.

\subsection{Kawazumi's results}
Instead of studying Leray spectral sequences, Kawazumi \cite{kawa} computed the twisted cohomology of $\M_{g,p}^b$ with $b\ge 1$ in a different way. He first introduced a notation called weighted partitions:
\begin{defn}[Kawazumi \protect{\cite[Defn.~1.1]{kawa}}]
A set $\hat{P}=\{(S_1,i_1),(S_2,i_2),\cdots,(S_\nu,i_\nu)\}$ is a \textbf{weighted partition} of the index set $\{1,2,\cdots,r\}$ if
\begin{enumerate} 
\item The set $\{S_1,S_2,\cdots,S_\nu\} $ is a partition of the set $\{1,2,\cdots,r\}$.
\item $i_1,i_2,\cdots,i_\nu \ge 0$ are non-negative integers.
\item Each $(S_a,i_a),1\le a \le \nu$, satisfies: $i_a+|S_a|\ge 2$.
\end{enumerate}
We denote by $\mathcal{P}_r$ the set of all weighted partitons of $\{1,2,\cdots,r\}$.
\end{defn}

Given a weighted partition $\hat{P}$, Kawazumi defined (\cite[eq.~1.10]{kawa}) the twisted Morita-Mumford class $$m_{\hat{P}}\in H^{2(\sum\limits_{a=1}^{\nu} i_a)+r-2\nu}(\M_g^1;H^1(\Sigma_g^1;\mathbb{Z})^{\otimes r}),$$
and then computed the twisted cohomology of $\M_{g,1}$ with coefficients in $H^1(\Sigma_g^1;\mathbb{Z})^{\otimes r}$:
\begin{thm}[Kawazumi \protect{\cite[Theorem 1.B]{kawa}}]\label{kawaB}
For degrees $\le N(g)-r$
$$H^\bullet (\M_g^1;H^1(\Sigma_g^1;\mathbb{Z})^{\otimes r})\cong \bigoplus_{\hat{P}\in \mathcal{P}_r} H^{\bullet } (\M_g^1;\mathbb{Z})m_{\hat{P}}.$$ 
\end{thm}

Kawazumi also computed the twisted cohomology of $\M_{g,p}^b$ with coefficients $H^1(\Sigma_{g,p}^b;\mathbb{Z})^{\otimes r}$ for general $b\ge 1$. His proof is based on the Gysin sequence (Proposition \ref{gy}) and induction.

\begin{thm}[Kawazumi \protect{\cite[Theoerm 1.A]{kawa}}]\label{ka}
For $b\ge1$, $ p\ge 0$, we have
$$H^\bullet(\M_{g,p}^b;H^1(\Sigma_{g,p}^b;\mathbb{Z})^{\otimes r})\cong H^\bullet (\M_{g,p}^b;\mathbb{Z})\otimes_{H^\bullet  (\M_g^1;\mathbb{Z})} H^\bullet(\M_g^1;H^1(\Sigma_g^1;\mathbb{Z})^{\otimes r})$$ in degrees $\le N(g)-r$. 
\end{thm}

\subsection{Re-prove the case of $\Sigma_{g}^1$} 

While Kawazumi’s Theorem \ref{kawaB} allows us to obtain the cohomology of $\M_g^1$ with coefficients in $H^1(\Sigma_g^1;\Q)^{\otimes r}$, we provide an alternative derivation using Looijenga’s framework.

We define a subspace of $A_r^{\bullet}$ from Theorem \ref{wowLooi} with mixed Hodge substructures:
\begin{equation}\label{arrr} A''^{\bullet}_r=\bigoplus_{P|[r]} (\prod_{\{i\}\in P}u_i) \Q[u_I:I\in P] a_P.\end{equation}
A reinterpretation of Theorem \ref{kawaB} using Looijenga's setup is:
\begin{prop}\label{prop3.4} Let $A''^{\bullet}_r$ be as defined in \eqref{arrr}. There is a graded map of $H^\bullet(\M_{g}^1;\Q)$-modules, which is also a morphism of mixed Hodge structures:
$$ H^\bullet(\M_{g}^1;\Q)\otimes A''^{\bullet}_r\to H^{\bullet-r}(\textup{Mod}_g^1;H^1(\Sigma_g^1;\Q)^{\otimes r}),$$
which is an isomorphism in degrees $k$ such that $g\ge \frac{3}{2}k+1$.
\end{prop}

A non-rigid proof is identifying the twisted Morita-Mumford class $m_{\hat{P}}$ with $\prod\limits_{S_a\in P} u_{S_a}^{i_a} \cdot a_P$, where $\hat{P}=\{(S_1,i_1),(S_2,i_2),\cdots,(S_\nu,i_\nu)\}$ and $P=\{S_1,S_2,\cdots,S_{\nu}\}$. It can be observed that their degrees defer by $r$, which is because Proposition \ref{prop3.4} is stated with a degree shift by $r$. To rigorously prove Proposition \ref{prop3.4}, we first need the following Lemma:
\begin{lem}\label{lem3.5}
 $H^0(\M_{g,1};H^2(\Sigma_g;\Q))$ is isomorphic to $\Q$ generated by $a_{\{1,2\}}$.
\end{lem}
\begin{proof}
Recall that $\mathcal{C}_{g,2}(\ell)$ is the moduli space of closed genus-$g$ Riemann surfaces with two ordered marked points that are not necessarily distinct. By forgetting the first marked point in $\mathcal{C}_{g,2}(\ell)$, we obtain a map $\mathcal{C}_{g,2}\to \mathcal{M}_{g,1}$. The associated Leray spectral sequence is 
$$E_2^{p,q}=H^p(\mathcal{M}_{g,1};H^q(\Sigma_g;\Q))\Rightarrow H^{p+q}(\mathcal{C}_{g,2};\Q).$$
We then identify $H^0(\M_{g,1};H^2(\Sigma_g;\Q))$ with $E_2^{0,2}$. This Leray spectral sequence degenerates at $E_2$ by Deligne's Theorem \ref{ever}, so $E_{g}=E_2$.  Therefore, the map 
$$H^k(\mathcal{C}_{g,2};\Q)\to E_{g}^{k-2,2}\to E_2^{k-2,2}=H^{k-2}(\mathcal{M}_{g,1};H^2(\Sigma_g;\Q))$$
is integration along fibers. Recall that $a_{\{1,2\}}$ denotes the Poincar\'e dual of $\mathcal{M}_{g,1}$ in $\mathcal{C}_{g,2}$, where $\mathcal{M}_{g,1}$ embeds into $\mathcal{C}_{g,2}$ via the trivial section: $$\mathcal{M}_{g,1}\to\mathcal{C}_{g,2}, (C,x_1)\mapsto (C,x_1,x_2=x_1).$$ Then for any $\omega\in H^{k-2}(\mathcal{M}_{g,1};H^2(\Sigma_g;\Q))$, letting $\tilde{\omega}$ be its preimage in $H^k(\mathcal{C}_{g,2};\Q)$, we have
$$\int_{\mathcal{C}_{g,2}} \tilde{\omega} \wedge a_{\{1,2\}}=\int_{\mathcal{M}_{g,1}} \omega.$$
 Thus $E_2^{0,2}=H^0(\M_{g,1};H^2(\Sigma_g;\Q))$ is isomorphic to $\Q$ generated by $a_{\{1,2\}}$.
\end{proof}

\begin{proof}[\rm\bf{Proof of Proposition \ref{prop3.4}}]
We first prove by induction on $r$ that there is an isomorphism in degrees $\le \frac{2}{3}(g-1)$:
$$ H^\bullet(\M_{g,1};\Q)\otimes A''^{\bullet}_r \to H^{\bullet-r}(\M_{g,1};H^1(\Sigma_{g,1};\Q)^{\otimes r}).$$
When $r=0$, the statement is vacantly true.

For $r\ge 1$, we suppose it is true for any $s \le r-1$ there is an isomorphism in degrees $\le \frac{2}{3}(g-1)$:
$$ H^\bullet(\M_{g,1};\Q)\otimes A''^{\bullet}_s \to H^{\bullet-s}(\M_{g,1};H^1(\Sigma_{g,1};\Q)^{\otimes s}).$$

Recall that $\mathcal{C}_{g,r+1}$ is the moduli space of closed genus-$g$ Riemann surfaces with $(r+1)$ ordered marked points. Looijenga's Theorem \ref{wowLooi} gives us the following isomorphism of mixed Hodge structures:
$$H^\bullet(\mathcal{C}_{g,r+1};\Q)\cong H^\bullet(\M_{g};\Q)\otimes A^{\bullet}_{r+1} $$ in degrees $\le \frac{2}{3}(g-1)$.

By forgetting all but the last marked point in $\mathcal{C}_{g,r+1}$, we obtain a map $f:\mathcal{C}_{g,r+1}\to \mathcal{M}_{g,1}$ where $\mathcal{C}_{g,r+1}$ is a complex orbifold. The associated Leray spectral sequence with $\Q$-coefficients is
$$ E^{p,q}_2=H^p (\mathcal{M}_{g,1};R^q f_*(\Q))\Rightarrow H^{p+q}(\mathcal{C}_{g,r+1};\Q).$$
Since we are working over $\Q$, we have $R^q f_*(\Q)=H^q((\Sigma_g)^{\times r};\Q)$.
Notice that each fiber of $f$ is projective. Thus we can apply the orbifold version of Deligne's Theorem \ref{ever}, which means the above spectral sequence degenerates at page $2$. This implies:
$$H^k(\mathcal{C}_{g,r+1};\Q)\cong \bigoplus_{p+q=k} H^p (\mathcal{M}_{g,1};H^q((\Sigma_g)^{\times r};\Q)).$$
The moduli space $\mathcal{M}_{g,1}$ of closed genus-$g$ Riemann surfaces with one marked point has the same rational cohomology as $\M_{g,1}$. Thus we can rewrite above as
$$H^k(\mathcal{C}_{g,r+1};\Q)\cong \bigoplus_{p+q=k} H^p (\M_{g,1};H^q((\Sigma_g)^{\times r};\Q)). $$
The Leray filtration preserves the mixed Hodge structure of $H^\bullet(\mathcal{C}_{g,r+1};\Q)$, therefore the $E_2$-page terms $E_2^{p,q}=H^p (\M_{g,1};H^q((\Sigma_g)^{\times r};\Q))$ inherit mixed Hodge structures. 

We can expand the coefficients $H^q((\Sigma_g)^{\times r};\Q)$ by the K\"unneth formula:
$$ H^q((\Sigma_{g})^{\times r};\Q)\cong \bigoplus_{i_1+i_2+\cdots+i_r=q} H^{i_1}(\Sigma_{g};\Q)\otimes H^{i_2}(\Sigma_{g};\Q) \otimes \cdots \otimes H^{i_r}(\Sigma_{g};\Q).$$
Identifying $H^\bullet(\Sigma_g;\Q)$ with $H^\bullet(\Sigma_{g,1};\Q)$ as $\M_{g,1}$ modules, we have
$$\begin{aligned} &H^p (\M_{g,1};H^q((\Sigma_g)^{\times r};\Q)) \\
\cong & \bigoplus_{i_1+i_2+\cdots+i_r=q} H^p(\M_{g,1};H^{i_1}(\Sigma_{g,1};\Q)\otimes H^{i_2}(\Sigma_{g,1};\Q) \otimes \cdots \otimes H^{i_r}(\Sigma_{g,1};\Q))\end{aligned}$$

Observe that $H^{k-r}(\M_{g,1};H^1(\Sigma_{g,1};\Q)^{\otimes r})$ is the component of $H^{k-r} (\M_{g,1};H^r((\Sigma_g)^{\times r};\Q))$ with $i_1=i_2=\cdots=i_r=1$. Let's think about what the remaining components of $H^{k-r} (\M_{g,1};H^r((\Sigma_g)^{\times r};\Q))$ are:

\begin{enumerate}
\item When some $i_j=0$, the cohomology $$H^p(\M_{g,1};H^{i_1}(\Sigma_{g,1};\Q)\otimes \cdots \otimes H^{i_j}(\Sigma_{g,1};\Q) \otimes \cdots \otimes H^{i_r}(\Sigma_{g,1};\Q))$$ is a component of $H^p (\M_{g,1};H^q((\Sigma_g)^{\times r};\Q))$ where $q\le2r-2$. For each $2\le i \le r+1$, define $\psi_i:\mathcal{C}_{g,r+1} \to \mathcal{C}_{g,r}$ to be the map that forgets the $i$-th marked point. Notice the map $f:\mathcal{C}_{g,r+1}\to \mathcal{M}_{g,1}$ factors through $\psi_i$, so we have the following diagram:
\begin{equation*}\begin{aligned}\xymatrix{(\Sigma_g)^{\times r} \ar[d] \ar[r] & \mathcal{C}_{g,r+1} \ar[r]^{f} \ar[d]^{\psi_i}  & \mathcal{M}_{g,1} \ar[d]^{id} \\ (\Sigma_g)^{\times (r-1)}\ar[r]& \mathcal{C}_{g,r} \ar[r]& \mathcal{M}_{g,1} }\end{aligned}.\end{equation*}
The map $\psi_i^*:H^\bullet(\mathcal{C}_{g,r};\Q)\to H^\bullet(\mathcal{C}_{g,r+1};\Q)$ induces maps between items in the two Leray spectral sequences. That is
$$ H^p (\M_{g,1};H^q((\Sigma_g)^{\times (r-1)};\Q)) \to H^p (\M_{g,1};H^q((\Sigma_g)^{\times r};\Q)).$$
When $q\le 2r-2$, the image of the above map is clear by the K\"unneth formula and the induction on $r$.

\item When some $i_j=2$, the cup product
\begin{equation*}\begin{aligned}
H^{p}(\M_{g,1};H^{i_1}(\Sigma_{g,1};\Q)\otimes \cdots \otimes H^{\widehat{i_{j}}}(\Sigma_{g,1};\Q) \otimes \cdots \otimes H^{i_r}(\Sigma_{g,1};\Q)) \\
\xymatrix{
\otimes H^0(\M_{g,1};H^{2}(\Sigma_{g,1};\Q)) \ar[d] \\ H^p(\M_{g,1};H^{i_1}(\Sigma_{g,1};\Q)\otimes \cdots \otimes H^{i_j}(\Sigma_{g,1};\Q) \otimes \cdots \otimes H^{i_r}(\Sigma_{g,1};\Q))}
\end{aligned}
\end{equation*}
turns out to be an isomorphism by direct computations.
The term $$H^{p}(\M_{g,1};H^{i_1}(\Sigma_{g,1};\Q)\otimes \cdots \otimes H^{\widehat{i_j}}(\Sigma_{g,1};\Q) \otimes \cdots \otimes H^{i_r}(\Sigma_{g,1};\Q))$$ is a component of $H^p (\M_{g,1};H^\bullet((\Sigma_g)^{\times (r-1)};\Q)) $, so it is known by induction. By Lemma \ref{lem3.5}, we know $H^0(\M_{g,1};H^{2}(\Sigma_{g,1};\Q))$ is $\Q$ generated by $a_{\{1,j+1\}}$. We also make use of the relations (\ref{ua1}) (\ref{ua2}) satisfied by $u_i,a_I$ to simplify the notation (e.g.\ write $a_{\{1,2\}} a_{\{1,3\}}=a_{\{1,2,3\}}$).
\end{enumerate}

The maps in (1) and (2) are morphisms of mixed Hodge structures. All the Hodge structures involved are polarizable, hence semi-simple by Theorem \ref{semi}. Therefore, after carefully writing terms of the above two types in terms of partitions $P$, we can exclude them to get $H^{k-r} (\M_{g,1};H^1(\Sigma_{g,1};\Q)^{\otimes r})$ in Table 1 as follows. (For the polynomials in the table, we mean the degree $k$ parts of them. For the $j_1,j_2\cdots$ indices in the table, they should be distinct and between $2$ and $r+1$. The order listed is by increasing $q$. As a shorthand, we denote $K=H^\bullet(\M_{g};\Q)$. We always have degree $k\le \frac{2}{3}(g-1)$.)

\footnotesize
\begin{longtable}{|c|c|c|}
\caption{Rational cohomology of $\mathcal{C}_{g,r+1}$ written in two ways}\\
\hline
\endfirsthead
\hline
\endlastfoot
$P|[r+1]$ & $H^k (\mathcal{C}_{g,r+1};\Q)$ &  $ \bigoplus\limits_{p+q=k} H^p (\M_{g,1};H^q((\Sigma_g)^{\times r};\Q))$ \\
\hline
$\{1\},\cdots,\{r+1\}$ & $K\otimes \Q[u_1,u_2,\cdots, u_{r+1}]$ & $K\otimes\Q[u_{r+1}]$ \\
 & & $K\otimes u_{j_1}\Q[u_{r+1},u_{j_1}]$ \\
 & & $K\otimes u_{j_1} \cdot u_{j_2} \Q[u_{r+1},u_{j_1},u_{j_2}] $\\
 & & \vdots \\
 & & $K\otimes u_{j_1} \cdots u_{j_{r-1}} \Q[u_{r+1},u_{j_1},\cdots,u_{j_{r-1}}]$\\
& & ? $\subset H^{k-r}(\M_{g,1};H^1(\Sigma_{g,1};\Q)^{\otimes r})$ \\
\hline
$\{r+1\},$ & $K\otimes \Q[u_{r+1},u_{I_2}]a_{I_2}$ & ? $\subset H^{k-r}(\M_{g,1};H^1(\Sigma_{g,1};\Q)^{\otimes r})$ \\
$I_2=\{1,\cdots,r\}$& & \\
\hline
$\{r+1\}$, &  $K\otimes \Q[u_I:I\in P]a_P$ & $K\otimes \Q[u_{r+1},u_I:I\in P,|I|\ge2]a_P$\\
 $P\setminus \{r+1\}$ not as above& & $K\otimes u_{j_1}\Q[u_{r+1},u_{j_1},u_I:\{j_1\}\in P, I\in P,|I|\ge 2]a_P$ \\
 & & (if $\sum\limits_{|I|\ge 2} |I| <r-1)$ \\
 & & \vdots \\
 & & $K \otimes u_{j_1}\cdots u_{j_m} \Q[u_{r+1},u_{j_1},\cdots,u_{j_m},u_I:$\\
& & $:\{j_1\},\cdots \{j_m\}, I\in P, |I|\ge 2]a_P $\\
 & & ($m+\sum\limits_{I\in P,|I|\ge 2}|I|=r-1$)\\
 & & ? $\subset H^{k-r}(\M_{g,1};H^1(\Sigma_{g,1};\Q)^{\otimes r})$ \\
\hline
 & & (*Take the degree $k-2(|I_1|-1)$ part of polynomials.)\\
$r+1\in I_1, |I_1|\ge 2$ & $K\otimes \Q[u_I:I\in P] a_P$  & $K\otimes \Q[ u_I:I\in P,|I|\ge 2] \prod\limits_{I\in P, I\neq I_1} a_I)$ \\
 & & $K\otimes u_{j_1}\Q[ u_{j_1},u_I:\{j_1\}\in P, I\in P,|I|\ge 2] \prod\limits_{I\in P, I\neq I_1} a_I $\\
 & &  $K\otimes u_{j_1}\cdots u_{j_n}\Q[ u_{j_1},\cdots, u_{j_n},u_I:$\\
 & & $:\{j_1\},\cdots,\{j_n\},I\in P,|I|\ge 2] \prod\limits_{I\in P, I\neq I_1} a_I $\\
 & & ($n+\sum\limits_{I\in P, |I|\ge 2}|I|=r+1$) \\
 & & (*The above is equivalent to the degree $k$ part of\\
& &  the polynomial multipled by $u_{I_1}^{|I_1|-1}$,\\
 & & thus equivalent to polynomials whose last term is $a_P$.)\\
\hline
\end{longtable}
\normalsize

From Table 1, we conclude that the following map
$$H^\bullet(\M_{g};\Q)\otimes\Q[u_{r+1}]\otimes \left( \bigoplus_{P|[r]} (\prod_{\{i\}\in P}u_i) \Q[u_I:I\in P] a_P\right) \to H^{\bullet-r}(\M_{g,1};H^1(\Sigma_{g,1};\Q)^{\otimes r})$$
is an isomorphism of mixed Hodge structures in degrees $\le \frac{2}{3}(g-1)$.

Notice that via the map $f:\mathcal{C}_{g,r+1}\to \mathcal{C}_{g,1}$, the image of $H^\bullet(\mathcal{C}_{g,1};\Q)$ in $H^\bullet(\mathcal{C}_{g,r+1})$ is exactly
$$H^\bullet(\M_{g};\Q)\otimes\Q[u_{r+1}].$$
Thus we can rewrite the result as
\begin{equation}\label{mg1} H^{\bullet-r}(\M_{g,1};H^1(\Sigma_{g,1};\Q)^{\otimes r})\cong H^\bullet(\M_{g,1};\Q)\otimes \left( \bigoplus_{P|[r]} (\prod_{\{i\}\in P}u_i) \Q[u_I:I\in P] a_P\right)\end{equation} in degrees $\le \frac{2}{3}(g-1)$. We finish the induction.

The last step is getting the twisted cohomology of $\M_g^1$. By gluing a punctured disk to $\Sigma_g^1$, we obtain a map
$$1\to \mathbb{Z}\to \M_g^1\to \M_{g,1}\to 1.$$
The associated Gysin sequence (Proposition \ref{gy}) with coefficients $H^1(\Sigma_{g,1};\Q)^{\otimes r}$ is
$$\begin{aligned} \cdots \to H^{\bullet-r-2}(\M_{g,1};H^1(\Sigma_{g,1};\Q)^{\otimes r}) \to H^{\bullet-r}(\M_{g,1};H^1(\Sigma_{g,1};\Q)^{\otimes r}) \to \\
\to H^{\bullet-r}(\M_g^1;H^1(\Sigma_{g,1};\Q)^{\otimes r}) \to H^{\bullet-r-1}(\M_{g,1};H^1(\Sigma_{g,1};\Q)^{\otimes r}) \cdots \end{aligned}.$$
Here the map $$H^{\bullet-r-2}(\M_{g,1};H^1(\Sigma_{g,1};\Q)^{\otimes r}) \to H^{\bullet-r}(\M_{g,1};H^1(\Sigma_{g,1};\Q)^{\otimes r})$$ is the multiplication by the Euler class $e_1\in H^2(\M_{g,1};\Q)$, so it is injective. Thus the above long exact sequence splits into short exact sequences:
$$\begin{aligned} 0\to  H^{\bullet-r-2}(\M_{g,1};H^1(\Sigma_{g,1};\Q)^{\otimes r}) \xrightarrow{\cdot e_1}&  H^{\bullet-r}(\M_{g,1};H^1(\Sigma_{g,1};\Q)^{\otimes r}) \\ \to& H^{\bullet-r}(\M_g^1;H^1(\Sigma_{g,1};\Q)^{\otimes r}) \to 0.\end{aligned}$$
The first two nonzero terms can be obtained from (\ref{mg1}). Taking cokernel in the short exact sequence, we have
$$H^{\bullet-r}(\M_g^1;H^1(\Sigma_g^1;\Q)^{\otimes r})\cong H^\bullet(\M_{g};\Q)\otimes \left( \bigoplus_{P|[r]} (\prod_{\{i\}\in P}u_i) \Q[u_I:I\in P] a_P\right)$$ in degrees $\le \frac{2}{3}(g-1)$.
\end{proof}

\subsection{Extending the results to general $\Sigma_{g,p}^b$}
By combining the results of Looijenga and Kawazumi, we establish the following theorem on the twisted cohomology of $\M_{g,p}^b$ for $p+b\ge 1$:

\begin{thm}\label{thmA}
Let $g,p,b$ be positive integers such that $p+b\ge 1$. The $\Q$-vector space $A''^{\bullet}_r$ is defined in \eqref{arrr}. There is a graded $S_r$-equivariant map of $H^\bullet(\M_{g,p}^b;\Q)$-modules, which is also a morphism of mixed Hodge structures:
$$ H^\bullet(\M_{g,p}^b;\Q)\otimes A''^{\bullet}_r\to H^{\bullet-r}(\textup{Mod}_{g,p}^b;H^1(\Sigma_{g,p}^b;\Q)^{\otimes r}),$$
which is an isomorphism in degrees $\le \frac{2}{3}(g-1)$.
\end{thm}

\begin{remark}
One could derive this result by repeating the method used to prove Proposition \ref{prop3.4}. However, that approach requires lots of repetitive computations. Instead, we present a simplified proof using Kawazumi’s results.
\end{remark}

\begin{proof}
Applying Kawazumi’s Theorem \ref{ka} over $\Q$, we obtain for $b\ge 1$:
$$
H^{\bullet }(\M_{g,p}^b;H^1(\Sigma_{g,p}^b;\Q)^{\otimes r}) \cong H^\bullet (\M_{g,p}^b;\mathbb{Z})\otimes_{H^\bullet  (\M_g^1;\mathbb{Z})} H^\bullet(\M_g^1;H^1(\Sigma_g^1;\mathbb{Z})^{\otimes r}).$$
Using the expression of $H^\bullet(\M_g^1;H^1(\Sigma_g^1;\Q)^{\otimes r})$ in Proposition \ref{prop3.4}, we have 
\begin{equation}\label{eq3.2}\begin{aligned}& H^{\bullet-r}(\M_{g,p}^b;H^1(\Sigma_{g,p}^b;\Q)^{\otimes r})\\
\cong & H^\bullet(\M_{g,p}^b;\Q)  \otimes_{H^\bullet(\M_{g}^1;\Q)} \left( H^\bullet(\M_{g}^1;\Q)\otimes A''^{\bullet}_r \right) \\
\cong & H^\bullet(\M_{g,p}^b;\Q)\otimes A''^{\bullet}_r
\end{aligned}\end{equation}
in degrees $\le \frac{2}{3}(g-1)$.

The remaining case to consider is when $b=0$ and $p\ge 1$. We proceed by gluing a punctured disk to the boundary component of $\Sigma_{g,p-1}^1$, which gives rise to the exact sequence: 
$$1\to \mathbb{Z}\to \M_{g,p-1}^1\to \M_{g,p}\to 1.$$
Taking the Gysin sequence (Proposition \ref{gy}) with coefficients $H^1(\Sigma_{g,p-1}^1;\Q)^{\otimes r}\cong H^1(\Sigma_{g,p};\Q)^{\otimes r}$, we obtain:
$$\begin{aligned}
\cdots\to H^{\bullet-r-2}(\M_{g,p};H^1(\Sigma_{g,p};\Q)^{\otimes r})\to H^{\bullet -r}(\M_{g,p};H^1(\Sigma_{g,p};\Q)^{\otimes r}) \to \\
 H^{\bullet -r}(\M_{g,p-1}^1;H^1(\Sigma_{g,p-1}^1;\Q)^{\otimes r}) \to H^{\bullet -r-1}(\M_{g,p};H^1(\Sigma_{g,p};\Q)^{\otimes r})\to \cdots.
\end{aligned}$$
Here, the map $$H^{\bullet -r-2}(\M_{g,p};H^1(\Sigma_{g,p};\Q)^{\otimes r})\to H^{\bullet -r}(\M_{g,p};H^1(\Sigma_{g,p};\Q)^{\otimes r})$$ is left multiplication by the Euler class $e_p\in H^2(\M_{g,p};\Q)$, which is injective in degrees $\le \frac{2}{3}(g-1)$. This allows us to split the above sequence into short exact sequences:
$$\begin{aligned} 0\to  H^{\bullet-r-2}(\M_{g,p};H^1(\Sigma_{g,p};\Q)^{\otimes r}) \xrightarrow{\cdot e_p} &   H^{\bullet -r}(\M_{g,p};H^1(\Sigma_{g,p};\Q)^{\otimes r})\\ \to& H^{\bullet -r}(\M_{g,p-1}^1;H^1(\Sigma_{g,p-1}^1;\Q)^{\otimes r}) \to 0.\end{aligned}$$
From this, we conclude:
$$H^{\bullet -r}(\M_{g,p};H^1(\Sigma_{g,p};\Q)^{\otimes r}) \cong \Q[e_p]\otimes H^{\bullet -r}(\M_{g,p-1}^1;H^1(\Sigma_{g,p-1}^1;\Q)^{\otimes r}).$$
Using equation (\ref{eq3.2}), we obtain:
$$\begin{aligned}&H^{\bullet -r}(\M_{g,p};H^1(\Sigma_{g,p};\Q)^{\otimes r}) \\
\cong & \Q[e_p] \otimes H^\bullet(\M_{g,p-1}^1;\Q)\otimes A''^{\bullet}_r  \\
\cong & H^\bullet(\M_{g,p}^b;\Q) \otimes A''^{\bullet}_r
\end{aligned}$$
in degrees $\le \frac{2}{3}(g-1)$, where the last isomorphism follows from Theorem \ref{wuwu}. This completes the proof.
\end{proof}

\section{Moduli Spaces of Riemann Surfaces with Level Structures }

In this section, we study a moduli space denoted $\mathcal{C}_{g,r}(\ell)$ and compute its rational cohomology for sufficiently large genus $g$.

\subsection{The moduli space $\mathcal{M}_{g,r}(\ell)$}
Recall that the moduli space $\mathcal{M}_{g,r}(\ell)$ of Riemann surfaces homeomorphic to $\Sigma_{g,r}$ with a level-$\ell$ structure is defined as
$$\mathcal{M}_{g,r}(\ell) \defeq \textup{Teich}(\Sigma_{g,r})/ \M_{g,r}(\ell),$$ where $$\textup{Teich}(\Sigma_{g,p})=\{\text{complex structures on }\Sigma_{g,p}\}/\text{homotopy}$$ is the Teichm\"uller space of $\Sigma_{g,p}$. When $\ell \ge 3$, the action of $ \M_{g,r}(\ell)$ on $\textup{Teich}(\Sigma_{g,r})$ is free (see e.g.\ \cite[Sec.~7.1]{AndyPicard}). Consequently $\mathcal{M}_{g,r}(\ell)$ is a quasi-projective variety whose cohomology coincides with the group cohomology of $\M_{g,r}(\ell)$. 

For $\ell=2$, the moduli space $\mathcal{M}_{g,r}(2)$ is an orbifold, but it admits a finite branched cover (e.g.\ $\mathcal{M}_{g,r}(4)$) that is a quasi-projective variety.

A more explicit description of $\mathcal{M}_{g,r}(\ell)$ is given by:
$$\left\{(C,x_1,x_2,\cdots,x_r,h)\middle|\begin{gathered}C \text{: closed Riemann surface of genus }g, \text{ }x_i \text{: distinct points on }C,\\ h: H_1(C\backslash  \{x_1,\cdots,x_r\};\mathbb{Z}/\ell)\to H_1(\Sigma_{g,r};\mathbb{Z}/\ell)   \text{ an isomorphism} \end{gathered}\right\}/\sim ,$$
where $$(C,x_1,x_2,\cdots,x_r,h)\sim (C',x'_1,x'_2,\cdots,x'_r,h')$$ if there exists a biholomorphism $f:C\to C'$ such that $f(x_i)=x_i'$ for $1\le i \le r$ and the induced map $f_*$ on $H_1(-;\mathbb{Z}/\ell)$ satisfies $ h' \circ f_* =h$.

To define $\mathcal{C}_{g,r}(\ell)$ as a moduli space containing $\mathcal{M}_{g,r}(\ell)$, we introduce an alternative perspective on $\mathcal{M}_{g,r}(\ell)$. Consider the map $\pi^{(r+1)}: \mathcal{M}_{g,r+1}(\ell) \to \mathcal{M}_{g,1}(\ell)$ defined by 
$$\pi^{(r+1)}:(C,x_1,\cdots,x_{r+1},h) \mapsto (C,x_1, \widehat{h}),$$
where $\widehat{h}: H_1(C\backslash  \{x_1\};\mathbb{Z}/\ell)\to H_1(\Sigma_{g,1};\mathbb{Z}/\ell)$ is the map $p_2\circ h \circ s$ as shown in the diagram:
$$\begin{aligned}\xymatrix{ H_1(C\backslash  \{x_1,\cdots,x_{r+1}\};\mathbb{Z}/\ell) \ar[r]^(0.6)h \ar[d]^{p_1} & H_1(\Sigma_{g,r+1};\mathbb{Z}/\ell) \ar[d]^{p_2} \\
H_1(C\backslash  \{x_1\};\mathbb{Z}/\ell) \ar@{-->}@/^1pc/[u]^s & H_1(\Sigma_{g,1};\mathbb{Z}/\ell) 
}\end{aligned},$$
where the vertical maps $p_1, p_2$ are obtained by forgetting all but the first marked points, while $s$ is an arbitrary section of $p_1$. The definition of $\widehat{h}$ is independent of the choice of the section $s$, since the homology class of a loop around the $i$-th puncture ($2\le i\le r+1$) maps to $0$ in $H_1(\Sigma_{g,1};\mathbb{Z}/\ell)$. It can be verified directly that $\pi^{(r+1)}: \mathcal{M}_{g,r+1}(\ell) \to \mathcal{M}_{g,1}(\ell)$ is well-defined. 

Let $\mathcal{D}=H_1(\Sigma_{g,1};\mathbb{Z}/\ell)$. Denote by $\widetilde{\Sigma_{g,1}}_{(\mathcal{D})}$  the regular $\mathcal{D}$-cover of $\Sigma_{g,1}$ associated with the homomorphism
$$\pi_1(\Sigma_{g,1})\to H_1(\Sigma_{g,1};\mathbb{Z}/\ell).$$ The fiber of the map $\pi^{(r+1)}$ can be interpreted as follows:

\begin{prop}\label{orbit}
The fiber of the map $\pi^{(r+1)}: \mathcal{M}_{g,r+1}(\ell) \to \mathcal{M}_{g,1}(\ell)$ is isomorphic to the \textbf{orbit configuration space}
 $$Conf^{\mathcal{D}}_{r}(\widetilde{\Sigma_{g,1}}_{(\mathcal{D})})=\{(y_2,y_3,\cdots,y_{r+1})|y_i\in \widetilde{\Sigma_{g,1}}_{(\mathcal{D})}; y_i \neq d\cdot y_j \text{ for all } d\in \mathcal{D},i\neq j  \},$$
which is the space of $r$ ordered points in $\widetilde{\Sigma_{g,1}}_{(\mathcal{D})}$ in different $\mathcal{D}$-orbits.
\end{prop}
\begin{proof}
We prove this by induction on $r$.

For $r=1$, the orbit configuration space $Conf^{\mathcal{D}}_{1} (\widetilde{\Sigma_{g,1}}_{(\mathcal{D})})$ is $\widetilde{\Sigma_{g,1}}_{(\mathcal{D})}$. For a given element in $\mathcal{M}_{g,1}(\ell) $: $$\left(C,x_1,h:H_1(C\backslash  \{x_1\};\mathbb{Z}/\ell)\to H_1(\Sigma_{g,1};\mathbb{Z}/\ell)\right),$$
its preimage under $\pi^{(2)}:\mathcal{M}_{g,2}(\ell)\to \mathcal{M}_{g,1}(\ell)$ consists of elements
$$\left(C,x_1,x_2,\widetilde{h}:H_1(C\backslash  \{x_1,x_2\};\mathbb{Z}/\ell)\to H_1(\Sigma_{g,2};\mathbb{Z}/\ell)\right)\in\mathcal{M}_{g,2}(\ell),$$
where $x_2\in C\backslash  \{x_1\}$ and $p_2 \circ \widetilde{h}\circ s=h$. For two such elements $(C,x_1,x_2,\widetilde{h})$ and $(C,x_1,x_2,\widetilde{h}')$, they differ by 
$$\widetilde{h}^{-1}\circ \widetilde{h}': H_1(C\backslash  \{x_1,x_2\};\mathbb{Z}/\ell) \to H_1(C\backslash  \{x_1,x_2\};\mathbb{Z}/\ell).$$
Fix a symplectic basis $\alpha_1,\beta_1,\cdots,\alpha_g,\beta_g$ of $H_1(C\backslash  \{x_1\};\mathbb{Z}/\ell)$, and let $\delta\in H_1(C\backslash  \{x_1,x_2\};\mathbb{Z}/\ell)$ denotes the homology class of the loop around the $x_2$. Then the matrix corresponding to $\widetilde{h}^{-1} \circ \widetilde{h}'$ under the basis $\alpha_1,\beta_1,\cdots,\alpha_g,\beta_g,\delta$ is
$$\begin{bmatrix}
I_{2g\times 2g} & 0_{2g \times 1} \\
\vec{v} & 0_{1\times 1}
\end{bmatrix},$$
where $\vec v\in H_1(C\backslash  \{x_1\};\mathbb{Z}/\ell)\cong \mathcal{D}$. A point $x_2\in C\backslash  \{x_1\}$ and a deck transformation $\vec v \in \mathcal{D}$ together determine a point $y_2$ in the regular $\mathcal{D}$-cover of $C\backslash\{x_1\}$. Thus the fiber of $\pi^{(2)}$ is homeomorphic to $\widetilde{\Sigma_{g,1}}_{(\mathcal{D})}$. We remark that the long exact sequence of homotopy groups associated to the fibration $\widetilde{\Sigma_{g,1}}_{(\mathcal{D})} \hookrightarrow \mathcal{M}_{g,2}(\ell)\to \mathcal{M}_{g,1}(\ell)$ when $\ell \ge 3$ is
$$0\to  \pi_1(\widetilde{\Sigma_{g,1}}_{(\mathcal{D})}) \to \M_{g,2}(\ell) \to \M_{g,1}(\ell)\to 0,$$
which is exactly the mod-$\ell$ Birman exact sequence.

For $r\ge 2$, assume that the fiber of $\pi^{(r)}: \mathcal{M}_{g,r}(\ell) \to \mathcal{M}_{g,1}(\ell)$ is 
$$Conf^{\mathcal{D}}_{r-1}(\widetilde{\Sigma_{g,1}}_{(\mathcal{D})})=\{(y_2,y_3,\cdots,y_{r})|y_i\in \widetilde{\Sigma_{g,1}}_{(\mathcal{D})};\text{if } i\neq j, \forall d\in \mathcal{D},  y_i \neq d\cdot y_j \}.$$
Our goal is to find the fiber of $\pi^{(r+1)}: \mathcal{M}_{g,r+1}(\ell) \to \mathcal{M}_{g,1}(\ell)$. Notice that by definition $\pi^{(r+1)}$ is the composition map $\pi^{(r)} \circ \widetilde{\pi}$, where $ \widetilde{\pi}: \mathcal{M}_{g,r+1}(\ell) \to \mathcal{M}_{g,r}(\ell)$ is defined in a similar way to $\pi^{(r+1)}$ by forgetting the marked point $x_{r+1}$ and projecting $h: H_1(C\backslash  \{x_1,\cdots,x_{r+1}\};\mathbb{Z}/\ell) \to H_1(\Sigma_{g,r+1};\mathbb{Z}/\ell)$ to $\widehat{h}:H_1(C\backslash  \{x_1,\cdots,x_{r}\};\mathbb{Z}/\ell) \to H_1(\Sigma_{g,r};\mathbb{Z}/\ell)$. By the same argument as $r=1$, the fiber of $\widetilde{\pi}$ is the regular $\mathcal{D}$-cover of $\Sigma_{g,r}$, which we denote by $\widetilde{\Sigma_{g,r}}_{(\mathcal{D})}$. Therefore the fiber of the composition $\pi^{(r)} \circ \widetilde{\pi}$ is 
$$\begin{aligned} & \{(y_1,y_2,\cdots,y_{r-1})\in Conf^{\mathcal{D}}_{r-1} (\widetilde{\Sigma_{g,1}}_{(\mathcal{D})}), y_r\in \widetilde{\Sigma_{g,r}}_{(\mathcal{D})}\},   \end{aligned}$$
which is exactly $Conf^{\mathcal{D}}_{r}(\widetilde{\Sigma_{g,1}}_{(\mathcal{D})})$. This completes the induction.
\end{proof}

Since the fiber of $\mathcal{M}_{g,r+1}(\ell)\to \mathcal{M}_{g,1}(\ell)$ can be described as an orbit configuration space of $r$ points on the regular $\mathcal{D}$-cover of $\Sigma_{g,1}$, we can express $\mathcal{M}_{g,r+1}(\ell)$ as follows:
\begin{cor}\label{cor4.2}
The moduli space $\mathcal{M}_{g,r+1}(\ell)$ can be described as
$$\left\{\left(C,x_1,y_2,\cdots,y_{r+1},h\right)\middle|\begin{gathered}C \text{ closed genus-$g$ Riemann surface, }x_1\in C, y_2,\cdots,y_{r+1}\in \widetilde{C}_{(\mathcal{D})},\\
y_2,\cdots,y_{r+1}\text { project to distinct points in }C\setminus \{x_1\},\\
h:  H_1(C\backslash  \{x_1\};\mathbb{Z}/\ell)\to H_1(\Sigma_{g,1};\mathbb{Z}/\ell)  \text{ is an isomorphism}
\end{gathered}
\right\}/\sim,$$
where $\widetilde{C}_{(\mathcal{D})}$ is the regular $\mathcal{D}$-cover of $C$ associated with the homomorphism
$$\pi_1(C,x_1)\to  H_1(C\backslash  \{x_1\};\mathbb{Z}/\ell)\xrightarrow{h} H_1(\Sigma_{g,1};\mathbb{Z}/\ell)=\mathcal{D}.$$
Two tuples $(C,x_1,y_2,\cdots,y_{r+1},h)$ and $(C',x'_1,y'_2,\cdots,y'_{r+1},h')$ are equivalent if there exists a biholomorphism $f:C\to C'$  satisfying 
\begin{itemize}
\item $f(x_1)=x'_1$,
\item the induced map $f_*$ on $H_1(-;\mathbb{Z}/\ell)$ satisfies $f_* \circ h=h'$,
\item the unique lifting of $f$ with $\widetilde{f}(y_1)=y'_1$ satisfies $\widetilde{f}(y_i)=y'_i$ for $2\le i \le r+1$.
\end{itemize}
\end{cor}

\subsection{The moduli space $\mathcal{C}_{g,r}(\ell)$}
In corollary \ref{cor4.2}, if we relax the condition on the marked points $y_2,\cdots,y_{r+1}$ and allow these points to be located anywhere in the regular $\mathcal{D}$-cover of $\Sigma_g$, we obtain a larger moduli space that contains $\mathcal{M}_{g,r+1}(\ell)$ as an open subvariety provided $\ell \ge 3$. Let $\mathcal{C}_{g,r+1}(\ell)$ denote this moduli space, defined as:
$$\left\{(C,x_1,y_2,\cdots,y_{r+1},h)\middle|\begin{gathered}C \text{ closed genus-$g$ Riemann surface, }x_1\in C, \ y_2,\cdots,y_{r+1}\in \widetilde{C}_{(\mathcal{D})},\\h:  H_1(C\backslash  \{x_1\};\mathbb{Z}/\ell)\to H_1(\Sigma_{g,1};\mathbb{Z}/\ell)  \text{ is an isomorphism}\end{gathered}\right\}/\sim,$$\
where the equivalence relation is the same as that in $\mathcal{M}_{g,r+1}(\ell)$.

\begin{remark} The moduli space $\mathcal{C}_{g,r+1}(\ell)$ is strictly larger than the $r$-fold fiber product of $\pi^{(2)}:\mathcal{M}_{g,2}(\ell)\to \mathcal{M}_{g,1}(\ell)$. This is because the definition of $\mathcal{C}_{g,r+1}(\ell)$ permits the marked points $y_2,\cdots,y_{r+1}$ on $\widetilde{C}_{(\mathcal{D})}$ to project to the same point as $x_1$, which is not allowed in the fiber product.
\end{remark}

To describe subvarieties of $\mathcal{C}_{g,r}(\ell)$ with different conditions on the marked points $y_2,\cdots,y_r$ on the regular $\mathcal{D}$-cover of $\Sigma_g$, we introduce the following notion:
\begin{defn}\label{def1}
Fix a group $\mathcal{D}$. A set $\widetilde{P}=\{(S_1, \vec d_1),(S_2,\vec d_2),\cdots,(S_\nu,\vec d_\nu)\}$ is called a \textbf{$\mathcal{D}$-weighted partition} of the index set $[r]=\{1,2,\cdots,r\}$, if 
\begin{enumerate}
\item The set $\{S_1,S_2,\cdots,S_\nu\} $ is a partition of the set $\{1,2,\cdots,r\}$.
\item For each $1\le a \le \nu$, there is an order in $S_a=\{i_1<i_2<\cdots<i_{|S_a|}\}$.
\item For each $1\le a \le \nu$, the element $\vec d_a$ is a tuple $(d_a^{(1)},d_a^{(2)},\cdots, d_a^{(|S_a|-1)})$, with $d_a^{(i)}\in \mathcal{D}$. By convention, $\vec d_a$ is empty if $|S_a|=1$.
\end{enumerate}
We denote by $\mathcal{P}_r^{\mathcal{D}}$ the set of all $\mathcal{D}$-weighted partitions of the index set $\{1,2,\cdots,r\}$. For $\widetilde{I}=(S,\vec d)\in \widetilde{P}$, define $|\widetilde{I}|$ to be $|S|$.
\end{defn}

We then apply this notion to denote subvarieties of $\mathcal{C}_{g,r}(\ell)$ in the following way:
\begin{notation}
For a $\mathcal{D}$-weighted partition $\widetilde{P}=\{(S_1, \vec d_1),(S_2,\vec d_2),\cdots,(S_\nu,\vec d_\nu)\}$ of the index set $[r]=\{1,2,\cdots,r\}$, we denote by $\mathcal{M}(\widetilde{P})$ the subvariety of $\mathcal{C}_{g,r}(\ell)$ where the marked points $y_1$ (lifting of $x_1$), $y_2,\cdots, y_{r}$ satisfy
\begin{itemize}
\item points indexed by elements in the same $S_a$ lie in the same $\mathcal{D}$-orbit, while points indexed by different $S_a$'s do not;
\item for each $1\le a\le \nu$, write $S_a=\{i_1^{(a)}<\cdots<i^{(a)}_{|S_a|}\},$ and  $\vec d_a=(d_a^{(1)},\cdots, d_a^{(|S|-1)}),$ the points satisfy
$$ y_{i^{(a)}_j} =  d_a^{(j-1)}  \cdot y_{i_1^{(a)}}, \text{ for } 2\le j \le |S_a|.$$
\end{itemize}
\end{notation}

A $\mathcal{D}$-weighted partition encodes two pieces of information about the positions of the marked points on the  regular $\mathcal{D}$-cover of $\Sigma_g$. First, it specifies whether any of the points project to the same point on $C$. Second, if they do, it indicates how these points differ under the $\mathcal{D}$-action.

By definition, each $\mathcal{M}(\widetilde{P})$ can be identified with $\mathcal{M}_{g,|\widetilde{P}|}(\ell)$, where $|\widetilde{P}|=\nu$. Considering all possible configurations of the marked points, we obtain a natural stratification of $\mathcal{C}_{g,r}(\ell)$ as a disjoint union of these subvarieties:
\begin{equation}\label{stra}
\mathcal{C}_{g,r}(\ell)=\coprod\limits_{\widetilde{P}\in \mathcal{P}_{r}^{\mathcal{D}}} \mathcal{M}_{g,|\widetilde{P}|}(\ell),\end{equation} where $\mathcal{P}_{r}^{\mathcal{D}}$ denotes the set of all $\mathcal{D}$-weighted partitions of $[r]$.

\subsection{Rational cohomology of $\mathcal{C}_{g,r}(\ell)$}
The stratification \eqref{stra} of $\mathcal{C}_{g,r}(\ell)$ allows us to compute the rational cohomology of $\mathcal{C}_{g,r}(\ell)$.

First, let us introduce the following cohomology classes in $H^{\bullet}(\mathcal{C}_{g,r}(\ell);\Q)$:
\begin{enumerate}
\item For $1 \le i \le r$, let $$v_i\in H^2(\mathcal{C}_{g,r}(\ell);\Q)$$ be the first Chern class of $\theta_i=g_i^*(\theta)$, where $\theta$ is the relative tangent sheaf of the universal curve $\pi:\mathcal{M}_{g,1}\to \mathcal{M}_g$, and
$$g_i:\mathcal{C}_{g,r}(\ell)\to  \mathcal{M}_{g,1}(\ell) \to \mathcal{M}_{g,1}$$ is the composition map where the first map is
$$\begin{aligned}
\mathcal{C}_{g,r}(\ell)\to & \mathcal{M}_{g,1}(\ell) \\
\left(C,x_1,y_2,\cdots,y_{r},h\right) \mapsto & (C, x_1,h) \text{ if } i=1, \\
\left(C,x_1,y_2,\cdots,y_{r},h\right) \mapsto & (C, \overline{y_{i}},h) \text{ if } i\ge2.
\end{aligned}$$
Here $\overline{y_{i}}$ is the image of $y_{i}$ under the covering map $\widetilde{C}_{(\mathcal{D})}\to C$.
One can also view $v_i$ as the pullback of the Euler class $e_1\in H^2(\mathcal{M}_{g,1};\Q)$ via $g_i$. The first Chern class of $\theta_i$ restricted to the open subvariety $\mathcal{M}_{g,r+1}(\ell)$ is exactly the Euler class $$e_i\in H^2(\M_{g,r+1}(\ell);\Q)\cong H^2(\M_{g,r+1};\Q).$$ 
\item For a $\mathcal{D}$-weighted partition $\widetilde{P}$ of the index set $[r]$, and $\widetilde{I}=(S,\vec d)\in \widetilde{P}$ with $|S|\ge 2$, write 
$$S=\{i_1<i_2<\cdots<i_{|S|}\},$$  $$\vec d=(d^{(1)},d^{(2)},\cdots, d^{(|S|-1)}).$$ 
Let $$a_{\widetilde{I}}\in H^{2|\widetilde{I}|-2}(\mathcal{C}_{g,r}(\ell);\Q)$$ be the Poincar\'e dual of the subvariety of $\mathcal{C}_{g,r}(\ell)$ whose points $y_1,y_2,\cdots y_{r}$ satisfy
\begin{equation}\label{condi}
y_{i_{j+1}}= d^{(j)}\cdot y_{i_1}, \text{ for }1\le j \le |S|-1.\end{equation}
Denote this subvariety by $\mathcal{C}_{g,r}(\ell)[\widetilde{I}]$.
\end{enumerate}

\begin{lem}The cohomology classes $v_i$ and $a_{\widetilde{I}}$ satisfy the relations
\begin{equation}\label{va1}
v_i a_{\widetilde{I}}=v_j a_{\widetilde{I}}, \text{ if } i,j\in S \text{ with }\widetilde{I}=(S,\vec d); 
\end{equation}
\begin{equation}\label{va2}
a_{\widetilde{I}} a_{\widetilde{J}} =v_i^{|\widetilde{I} \cap \widetilde{J}|-1} a_{\widetilde{I}\cup \widetilde{J}}, \text{ if } i\in S_0 \text{ with } \widetilde{I} \cap \widetilde{J}=(S_0,\vec d_0)\neq \emptyset;
\end{equation}
\begin{equation}\label{va3}
a_{\widetilde{I}} a_{\widetilde{J}}= 0 \text{ if conditions (\ref{condi}) for $\widetilde{I}$ and $\widetilde{J}$ contradict.} 
\end{equation}
\end{lem}

\begin{proof}
First, for any $\widetilde{I}=(S,\vec d)\in \widetilde{P}$, if $i,j\in S$, then since $\theta_i=g_i^*(\theta)$ and $\theta_j=g_j^*(\theta)$ have isomorphic restrictions to $\mathcal{C}_{g,r}(\ell)[\widetilde{I}]$, we obtain the relation \eqref{va1}.

Next, consider two elements $\widetilde{I}=(S,\vec d)$ and $\widetilde{J}=(T,\vec p)$, the intersection of their corresponding subvarieties, $\mathcal{C}_{g,r+1}(\ell)[\widetilde{I}]\cap\mathcal{C}_{g,r+1}(\ell)[\widetilde{J}]$, falls into one of the following cases:
\begin{itemize}
\item If $S\cap T=\emptyset$, the intersection $\mathcal{C}_{g,r+1}(\ell)[\widetilde{I}]\cap \mathcal{C}_{g,r+1}(\ell)[\widetilde{J}]$ is a non-empty closed subvariety of $\mathcal{C}_{g,r+1}(\ell)$ that combines the conditions (\ref{condi}) imposed by $\widetilde{I}$ and $\widetilde{J}$. In this case, no additional relations hold between $a_{\widetilde{I}}$ and $a_{\widetilde{J}}$.
\item If $S\cap T\neq \emptyset$ and the conditions (\ref{condi}) imposed by $\widetilde{I}$ and $\widetilde{J}$ are compatible, we define the union of $\widetilde{I}$ and $\widetilde{J}$ as
$$\widetilde{I} \cup \widetilde{J}\defeq (S\cup T, \vec d \cup \vec p).$$
In this case, we have
$$ \mathcal{C}_{g,r+1}(\ell)[\widetilde{I}]\cap \mathcal{C}_{g,r+1}(\ell)[\widetilde{J}]=\mathcal{C}_{g,r+1}(\ell)[\widetilde{I} \cup \widetilde{J}].$$
Applying Lemma 2.4 from Looijenga \cite{LooiStable}, we obtain the relation (\ref{va2}).
\item If $S\cap T\neq \emptyset$ but the conditions \ref{condi} imposed by $\widetilde{I}$ and $\widetilde{J}$ are inconsistent, then $$\mathcal{C}_{g,r+1}(\ell)[\widetilde{I}]\cap \mathcal{C}_{g,r+1}(\ell)[\widetilde{J}]=\emptyset.$$ Consequently, we obtain the vanishing relation (\ref{va3}).
\end{itemize}
This completes the proof.
\end{proof}

\begin{ex}\label{ex4.5} Let us look at some examples of these relations.
\begin{enumerate}
\item If $\widetilde{I}=(\{1<2\},d)$, then $1,2 \in \widetilde{I}$ and the first relation in the above lemma applies:
$$ v_1 a_{(\{1<2\},d)}= v_2 a_{(\{1<2\},d)}.$$
\item If $\widetilde{I}=(\{1<2\},d)$ and $\widetilde{J}=(\{1<3\},d')$, then $\widetilde{I}\cup \widetilde{J}=(\{1<2<3\},(d,d'))$, because $$\mathcal{C}_{g,r+1}(\ell)[\widetilde{I}]\cap \mathcal{C}_{g,r+1}(\ell)[\widetilde{J}]=\mathcal{C}_{g,r+1}(\ell)[\widetilde{I}\cup \widetilde{J}].$$
The intersection is $\widetilde{I}\cap \widetilde{J}=(\{1\})$. The second relation above applies:
$$a_{(\{1<2\},d)}\cdot a_{(\{1<3\},d')}=1\cdot a_{(\{1<2<3\},(d,d'))}.$$
\item If $\widetilde{I}=(\{1<2\},d)$ and $\widetilde{J}=(\{1<2\},d')$ with $d\neq d'$, then we can see 
$$ \mathcal{C}_{g,r+1}(\ell)[\widetilde{I}]\cap \mathcal{C}_{g,r+1}(\ell)[\widetilde{J}]=\emptyset,$$
since you can not simultaneously require that the second marked point is $d\cdot x_1$ and $d'\cdot x_1$. In this case we should have $\widetilde{I}\cap  \widetilde{J}=\emptyset$ and we do not have the second relation above for $a_{\widetilde{I}}$ and $a_{\widetilde{J}}$. Instead, since $\mathcal{C}_{g,r+1}(\ell)[\widetilde{I}]$ and $\mathcal{C}_{g,r+1}(\ell)[\widetilde{J}]$ are disjoint, we have
$$ a_{(\{1<2\},d)} a_{(\{1<2\},d')}=0, \text{ if } d\neq d'.$$
\end{enumerate}
\end{ex}

Now we state the result of the rational cohomology of $\mathcal{C}_{g,r}(\ell)$.
\begin{thm}\label{coh2}
Let $v_i,a_{\widetilde{I}}$ be as defined above. Define $A_{r}(\ell)^{\bullet}$ as the graded commutative $\Q$-module generated by all $v_i$ and $a_{\widetilde{I}}$, subject to the relations (\ref{va1}), (\ref{va2}), and (\ref{va3}).			
Then, there exists an algebra homomorphism:
$$H^\bullet(\M_{g}(\ell);\Q)\otimes A_{r}(\ell)^{\bullet}  \to H^\bullet (\mathcal{C}_{g,r}(\ell);\mathbb{Q})$$
which is an isomorphism in degrees $k$ such that $g\ge 2k^2+7k+2$. 
\end{thm}

\begin{remark}
To express $A_{r}(\ell)^{\bullet}$ explicitly as a vector space, following Looijenga's approach in Remark \ref{rmk1}, we introduce the following notation:
\begin{itemize}
\item Given a $\mathcal{D}$-weighted partition $\widetilde{P}$ of the index set $[r]$, define:
$$a_{\widetilde{P}}=\prod\limits_{\widetilde{I}\in \widetilde{P},|\widetilde{I}|\ge 2} a_{\widetilde{I}}.$$
By relations (\ref{va2}) and (\ref{va3}), the module $A_{r}(\ell)^{\bullet}$ can be viewed as a $\Q[v_i:1\le i \le r]$-module generated by all $a_{\widetilde{P}}$, subject to the relation (\ref{va1}).
\item Given a $\mathcal{D}$-weighted partition $\widetilde{P}$ of $[r]$, define an equivalence relation on the set $\{v_1,\cdots,v_r\}$ by
$$v_i\sim v_j \text{ if and only if }i,j\in S \text{ for some } \widetilde{I}=(S,\vec d)\in \widetilde{P}.$$
Let $v_{\widetilde{I}}$ denote the equivalence class of $v_i$ for any $i\in S$ with $\widetilde{I}=(S,\vec d)\in \widetilde{P}$. In particular, when $\widetilde{I}=(\{i\},\emptyset)$, we have $v_I=v_i$.
\end{itemize}
Using this notation, the vector space $A_{r}(\ell)^{\bullet}$ admits the decomposition:
\begin{equation}\label{arl} A_{r}(\ell)^{\bullet}=\bigoplus_{\widetilde{P}\in \mathcal{P}_{r}^{\mathcal{D}}} \Q[v_{\widetilde{I}}:\widetilde{I}\in \widetilde{P}]a_{\widetilde{P}}.\end{equation}
\end{remark}

\begin{remark}\label{Cmix}
Fix the grading of the generators as follows:
\begin{itemize}
\item $v_i$ has degree $2$,
\item $a_{\widetilde{I}}$ has degree $2|\widetilde{I}|-2$. 
\end{itemize}
Then $A_{r}(\ell)^\bullet$ decomposes as
$$A_{r}(\ell)^\bullet=\bigoplus\limits_{m=0}^g  A_{r}(\ell)^{2m}$$ where $ A_{r}(\ell)^{2m}$ denotes the degree $2m$ part. This grading induces a mixed Hodge structure on $A_{r}(\ell)^\bullet$, in which $ A_{r}(\ell)^{2m}$ has Hodge type $(m,m)$.

 By Putman's Theorem \ref{andy0}, we have an isomorphism
$$H^\bullet(\M_{g}(\ell);\Q)\cong H^\bullet(\M_{g};\Q)$$ 
in degrees $k$ such that $g\ge 2k^2+7k+2$. Consequently, $H^\bullet(\M_{g}(\ell);\Q)$ inherits a canonical mixed Hodge structure as described in \cite[Sec.~2.5]{LooiStable}. 

Moreover, by Theorem \ref{polar}, the rational cohomolgy of $\mathcal{C}_{g,r}(\ell)$ carries a canonical polarizable mixed Hodge structure. In Theorem \ref{coh2}, the map 
$$H^\bullet(\M_{g}(\ell);\Q)\otimes A_{r}(\ell)^{\bullet}  \to H^\bullet (\mathcal{C}_{g,r}(\ell);\mathbb{Q})$$ 
is a morphism of mixed Hodge structures. 
\end{remark}

\begin{proof}[\rm\bf{Proof of Theorem \ref{coh2}}]
We compute $H^\bullet(\mathcal{C}_{g,r}(\ell);\Q)$ based on the stratification \eqref{stra}.

For integers $k\ge 0$, let $U_k$ be the union of the strata $\mathcal{M}_{g,|\widetilde{P}|}$ of codimension $\le 2k$, and let $S_k$ be the union of the strata $\mathcal{M}_{g,|\widetilde{P}|}$ of codimension $ 2k$. We prove by induction on $k$ that 
$$H^\bullet(\M_{g}(\ell);\Q)\otimes  \bigoplus_{\widetilde{P}\in \mathcal{P}_{r}^{\mathcal{D}}, r-|\widetilde{P}|\le k} \Q[v_{\widetilde{I}}:\widetilde{I}\in \widetilde{P}]a_{\widetilde{P}}\to H^\bullet(U_k;\Q)$$
is an isomorphism in degrees $k$ such that $g\ge 2k^2+7k+2$. 

The final case $k=r$ is what the theorem statement is since $U_{r}=\mathcal{C}_{g,r}(\ell)$.

In the base case $k=0$, the space $U_0$ is exactly $\mathcal{M}_{g,r}(\ell)$ with the weighted partition being $\widetilde{P}=\{\{1\},\{2\},\cdots,\{r\}\}.$
In degrees $k$ such that $g\ge 2k^2+7k+2$, by Theorem \ref{andy0} we have 
$$H^{\bullet}(\M_{g}(\ell);\Q)\otimes \Q[v_1,\cdots,v_{r}] \xrightarrow{\cong} H^\bullet(U_0;\Q)\cong H^{\bullet}(\M_{g};\Q)\otimes \Q[e_1,\cdots,e_{r}].$$
For general $k$, suppose the map 
\begin{equation}\label{rminus} H^\bullet(\M_{g}(\ell);\Q)\otimes  \bigoplus_{\widetilde{P}\in \mathcal{P}_{r}^{\mathcal{D}}, r-|\widetilde{P}|\le k-1} \Q[v_{\widetilde{I}}:\widetilde{I}\in \widetilde{P}]a_{\widetilde{P}}\to H^\bullet(U_{k-1};\Q)\end{equation} is an isomorphism in degrees $k$ such that $g\ge 2k^2+7k+2$.
Notice that $U_{k-1}$ is an open subvariety of $U_k$, whose complement $S_k$ has codimension $2k$. The corresponding Thom-Gysin sequence (Proposition \ref{tg}) is
$$\cdots\to H^{\bullet-2k}(S_k;\Q)\to H^\bullet(U_k;\Q)\to H^\bullet(U_{k-1};\Q)\to H^{\bullet+1}(U_k;\Q)\to \cdots.$$
Here the map $H^\bullet(U_k;\Q)\to H^\bullet(U_{k-1};\Q)$ is surjective since the map (\ref{rminus}) factorizes over $H^\bullet(U_k;\Q)$. Thus we have a splitting short exact sequence
$$0\to H^{\bullet-2k}(S_k;\Q)\to H^\bullet(U_k;\Q)\to H^\bullet(U_{k-1};\Q)\to 0.$$
Recall $$S_k=\coprod\limits_{\widetilde{P}\in \mathcal{P}_{r+1}^{\mathcal{D}},|\widetilde{P}|=r-k} \mathcal{M}_{g,|\widetilde{P}|}(\ell).$$
We know $H^\bullet(\mathcal{M}_{g,|\widetilde{P}|}(\ell);\Q)$ from Theorem \ref{andy0}, and the Gysin map $H^{\bullet-2k}(S_k;\Q)\to H^\bullet(U_k;\Q)$ restricted to the component $H^\bullet(\mathcal{M}_{g,|\widetilde{P}|}(\ell);\Q)$ is the multiplication by $a_{\widetilde{P}}\in H^{2k}(U_k;\Q)$. Thus we conclude that in degrees $k$ such that $g\ge 2k^2+7k+2$ we have
$$\begin{aligned} H^\bullet(U_{k};\Q)\cong& H^\bullet(U_{k-1};\Q)\oplus H^{\bullet-2k}(S_k;\Q) \\
\cong & H^\bullet(\M_{g}(\ell);\Q)\otimes  \bigoplus_{\widetilde{P}\in \mathcal{P}_{r}^{\mathcal{D}}, r-|\widetilde{P}|\le k-1} \Q[v_{\widetilde{I}}:\widetilde{I}\in \widetilde{P}]a_{\widetilde{P}} \\
& \oplus H^\bullet(\M_{g}(\ell);\Q)\otimes  \bigoplus_{\widetilde{P}\in \mathcal{P}_{r}^{\mathcal{D}}, r-|\widetilde{P}|= k} \Q[v_{\widetilde{I}}:\widetilde{I}\in \widetilde{P}]a_{\widetilde{P}} \\
\cong& H^\bullet(\M_{g}(\ell);\Q)\otimes  \bigoplus_{\widetilde{P}\in \mathcal{P}_{r}^{\mathcal{D}}, r-|\widetilde{P}|\le k} \Q[v_{\widetilde{I}}:\widetilde{I}\in \widetilde{P}]a_{\widetilde{P}}
.\end{aligned}$$
This finishes the induction.
\end{proof}

\section{Twisted Cohomology of the Level-$l$ Mapping Class Groups}

In this section, we prove Theorem \ref{main2} by embedding the twisted cohomology of the level-$\ell$ mapping class group, with coefficients in the $r$-tensor power of Prym representation, into $H^\bullet(\mathcal{C}_{g,r}(\ell);\Q)$. 

From Theorem \ref{coh2}, we recall that the rational cohomology of $\mathcal{C}_{g,r}(\ell)$ is described using an algebra $A_r(\ell)^{\bullet}$, explicitly given by \eqref{arl}.  To proceed, we define a $\Q$-subspace of $A_r(\ell)^{\bullet}$ as
\begin{equation}\label{mainsubspace}
A'_r(\ell)^{\bullet}=\bigoplus_{\widetilde{P}\in\mathcal{P}_r^{\mathcal{D}}}(\prod_{\{i\}\in \widetilde{P}} v_i) \Q[v_{\widetilde{I}}:\widetilde{I}\in \widetilde{P}]a_{\widetilde{P}},
\end{equation}
where $\mathcal{P}_r^{\mathcal{D}}$ is the set of all $\mathcal{D}$-weighted partition of the index set $[r]$ (see Definition \ref{def1}). 

The symmetric group $S_r$ naturally acts on $A'_r(\ell)^{\bullet}$ by permuting the indices of $[r]$. Furthermore, $A'_r(\ell)^{\bullet}$ inherits a mixed Hodge substructure from the mixed Hodge structure of $A_r(\ell)^{\bullet}$ in Remark \ref{Cmix}.

\subsection{Case of $\M_{g,1}(\ell)$} We first analyze the case of $\M_{g,1}(\ell)$. 
\begin{thm}\label{case1}
There is a graded $S_r$-equivariant map of $H^{\bullet}(\M_{g,1}(\ell);\mathbb{Q})$-modules, which is also a morphism of mixed Hodge structures:
$$H^\bullet (\M_{g,1}(\ell);\Q)\otimes A'_r(\ell)^{\bullet} \to H^{\bullet-r}(\M_{g,1}(\ell);\fH_{g,1}(\ell;\Q)^{\otimes r})$$
which is an isomorphism in degrees $k$ such that $g\ge 2k^2+7k+2$.
\end{thm}

Recall that $\mathcal{D}=H_1(\Sigma_g;\mathbb{Z}/\ell)$ and $\fH_{g,1}(\ell;\Q)=H^1(\widetilde{\Sigma_{g,1}}_{(\mathcal{D})};\Q)$ where $\widetilde{\Sigma_{g,1}}_{(\mathcal{D})}$ denotes the regular $\mathcal{D}$-cover of $\Sigma_{g,1}$. Let $\widetilde{\Sigma_{g}}_{(\mathcal{D})}$ denote the regular $\mathcal{D}$-cover of $\Sigma_{g}$.

To compute $H^{\bullet}(\M_{g,1}(\ell);\fH_{g,1}(\ell;\Q)^{\otimes r})$, we also need to consider twisted cohomology with coefficients in $$H^1(\widetilde{\Sigma_g}_{(\mathcal{D})};\Q)^{\otimes r}, H^1(\widetilde{\Sigma_g}_{(\mathcal{D})};\Q)^{\otimes r-1}\otimes \fH_{g,1}(\ell;\Q),$$ and similar mixed tensor products. To handle these cases systematically, we introduce the following indexing notation:
\begin{defn}
Define $f(1)=\fH_{g,1}(\ell;\Q)$ and $f(0)=H^1(\widetilde{\Sigma_g}_{(\mathcal{D})};\Q)$. For any $r\ge 1$, let $J=(J_1,J_2,\cdots,J_r)$ be a sequence where $J_i\in \{0,1\}$ for all $i$. Then, we define the corresponding $r$-tensor product:
$$ \fH^r (J) \defeq f(J_1)\otimes f(J_2) \otimes \cdots  \otimes f(J_r).$$ 
\end{defn}

This notation allows us to track different tensor products involving $H^1(\widetilde{\Sigma_g}_{(\mathcal{D})};\Q)$ and $\fH_{g,1}(\ell;\Q)$ in a structured manner.

To describe the cohomology of $\M_{g,1}(\ell)$ with coefficients in $\fH^r (J) $, we define when a $\mathcal{D}$-weighted partition is compatible with $J$:

\begin{defn}\label{compa}
Given $\mathcal{D}$, $r$, $J$, $\fH^r(J)$ as above, a $\mathcal{D}$-weighted partition $$\widetilde{P}=\{(S_1, \vec d_1),(S_2,\vec d_2),\cdots,(S_\nu,\vec d_\nu)\}$$ indexed by $[r+1]$ is said to be \textbf{compatible with $J$} if:
\begin{enumerate}
\item By convention, we assume $1\in S_1$. 
\item For $\vec d_1=(d_1^{(1)},d_1^{(2)},\cdots,d_1^{(|S_1|-1)})$, each $d_1^{(i)}$ ($1\le i \le  |S_1|-1$) is not the unit $1$ in $\mathcal{D}$. 
\item $S_1$ does not contain any $2\le a\le r+1$ such that $J_{a-1}=1$.
\end{enumerate}
We denote by $\mathcal{P}^{\mathcal{D}}_{r+1}(J)$ the set of all $\mathcal{D}$-weighted partitions indexed by $[r+1]$ which are compatible with $J$.
\end{defn}

Using this notation, the twisted cohomology $H^{\bullet}(\M_{g,1}(\ell);\fH^r(J))$ is described as follows:
\begin{thm}\label{gcase}
There is a graded map of $H^{\bullet}(\M_{g}(\ell);\mathbb{Q})$-modules:
$$ H^\bullet (\M_{g}(\ell);\Q)\otimes \left( \bigoplus_{\widetilde{P}\in \mathcal{P}^{\mathcal{D}}_{r+1}(J)} ( \prod_{\{i\}\in \widetilde{P},i\neq 1} v_i )\Q[v_{\widetilde{I}}:\widetilde{I}\in \widetilde{P}] a_{\widetilde{P}}\right)\to H^{\bullet-r}(\M_{g,1}(\ell);\fH^r(J))$$ 
which is an isomorphism in degrees $k$ such that $g \ge 2k^2+7k+2$. 
\end{thm}

\begin{remark}
Theorem \ref{case1} follows as a special case of Theorem \ref{gcase} when $J=(1,1,\cdots,1)$. By Definition \ref{compa}, a $\mathcal{D}$-weighted partition indexed by $[r+1]$ that is compatible with $(1,1,\cdots,1)$ takes the form
$$\widetilde{P}=\{(S_1, \vec d_1)=(\{1\},\emptyset),(S_2,\vec d_2),\cdots,(S_\nu,\vec d_\nu)\}.$$
This is in one-to-one correspondence with a $\mathcal{D}$-weighted partition indexed by $[r]$:
$$\widetilde{P}'=\{(S_2,\vec d_2),\cdots,(S_\nu,\vec d_\nu)\}.$$
Thus, when $J=(1,1,\cdots,1)$, we obtain
$$\begin{aligned} &H^\bullet (\M_{g}(\ell);\Q)\otimes \left(\bigoplus_{\widetilde{P}\in \mathcal{P}^{\mathcal{D}}_{r+1}(J)} ( \prod_{\{i\}\in \widetilde{P},i\neq 1} v_i )\Q[v_{\widetilde{I}}:\widetilde{I}\in \widetilde{P}] a_{\widetilde{P}} \right)\\
\cong & H^\bullet (\M_{g}(\ell);\Q)\otimes \left( \bigoplus_{\widetilde{P}'\in \mathcal{P}^{\mathcal{D}}_{r}} ( \prod_{\{i\}\in \widetilde{P}'} v_i )\Q[v_1,v_{\widetilde{I}}:\widetilde{I}\in \widetilde{P}'] a_{\widetilde{P}'} \right) \\
\cong & H^\bullet (\M_{g}(\ell);\Q)\otimes \Q[v_1] \otimes  \left(\bigoplus_{\widetilde{P}'\in \mathcal{P}^{\mathcal{D}}_{r}} ( \prod_{\{i\}\in \widetilde{P}'} v_i )\Q[v_{\widetilde{I}}:\widetilde{I}\in \widetilde{P}'] a_{\widetilde{P}'} \right) \\
\cong & H^\bullet (\M_{g,1}(\ell);\Q) \otimes  \left(\bigoplus_{\widetilde{P}'\in \mathcal{P}^{\mathcal{D}}_{r}} ( \prod_{\{i\}\in \widetilde{P}'} v_i )\Q[v_{\widetilde{I}}:\widetilde{I}\in \widetilde{P}'] a_{\widetilde{P}'} \right).
\end{aligned}$$
\end{remark}

Now it suffices to prove Theorem \ref{gcase}. We begin with proving the case $J=(0,0,\cdots,0)$, where the coefficients $\fH^r(J)=H^1( \widetilde{\Sigma_g}_{(\mathcal{D})};\Q)^{\otimes r}$:
\begin{proof}[\rm\bf{Proof of Theorem \ref{gcase} when $J=(0,0,\cdots,0)$}]
We proceed by induction on $r$. For $r=0$, the statement is 
$$H^\bullet (\M_{g,1}(\ell);\Q) \cong H^\bullet (\M_{g}(\ell);\Q) \otimes \Q[v_1] ,$$
which follows from Putman's Theorem \ref{andy0} and Looijenga's Theorem \ref{wuwu}.

Let $r\ge 1$. In the previous section, we studied the moduli space $\mathcal{C}_{g,r+1}(\ell)$ and computed its stable rational cohomology.  Define the map 
$$\begin{aligned} \pi: \mathcal{C}_{g,r+1}(\ell) &\to \mathcal{M}_{g,1}(\ell) &\\
(C,x_1,y_2,\cdots,y_{r+1},h) &\mapsto (C,x_1,h),\end{aligned}$$
which is a holomorphic map of quasi-projective varieties when $\ell\ge 3$ and a map of orbifolds when $\ell=2$.

The fiber of $\pi$ consists of $r$ ordered (not necessarily distinct) points in the regular $\mathcal{D}$-cover of the closed genus-$g$ Riemann surface $C$, so the fiber is homeomorphic to $(\widetilde{\Sigma_g}_{(\mathcal{D})})^{\times r}$. Since the fiber is projective, we can apply Deligne's degeneration theorem \ref{ever} to the map $f$, implying that the associated Leray spectral sequence degenerates at the second page: 
$$E_2^{p,q}=H^p( \mathcal{M}_{g,1}(\ell);H^q((\widetilde{\Sigma_g}_{(\mathcal{D})})^{\times r};\Q))\Rightarrow H^{p+q}(\mathcal{C}_{g,r+1}(\ell);\Q).$$
Thus we have:
$$H^{k}(\mathcal{C}_{g,r+1}(\ell);\Q) \cong \bigoplus_{p+q=k}H^p( \mathcal{M}_{g,1}(\ell);H^q(( \widetilde{\Sigma_g}_{(\mathcal{D})})^{\times r};\Q)).$$
Since the Leray filtration respects the mixed Hodge structure of $H^*(\mathcal{C}_{g,r+1}(\ell);\Q)$, terms on the second page inherit mixed Hodge structures.

On the one hand, by Theorem \ref{coh2}, the rational cohomology of $\mathcal{C}_{g,r+1}(\ell)$ is given by
$$\begin{aligned} H^\bullet(\mathcal{C}_{g,r+1}(\ell);\Q) \cong \bigoplus_{\widetilde{P}\in \mathcal{P}_{r+1}^{\mathcal{D}}} H^\bullet (\M_{g}(\ell);\Q)\otimes \Q[v_{\widetilde{I}}:\widetilde{I}\in \widetilde{P}]a_{\widetilde{P}}
\end{aligned}$$
in degrees $k$ such that $g\ge 2k^2+7k+2$. This is also an isomorphism of mixed Hodge structures, as discussed in Remark \ref{Cmix}.

On the other hand, we can expand the terms on the second page of the Leray spectral sequence using the K\"unneth formula:
$$\begin{aligned} &H^p (\M_{g,1}(\ell);H^q((\widetilde{\Sigma_g}_{(\mathcal{D})})^{\times r};\Q)) \\
\cong & \bigoplus_{i_1+i_2+\cdots+i_r=q} H^p(\M_{g,1}(\ell);H^{i_1}(\widetilde{\Sigma_g}_{(\mathcal{D})};\Q)\otimes H^{i_2}(\widetilde{\Sigma_g}_{(\mathcal{D})};\Q) \otimes \cdots \otimes H^{i_r}(\widetilde{\Sigma_g}_{(\mathcal{D})};\Q)).\end{aligned}$$
Our target $H^{k-r}(\M_{g,1}(\ell);H^1( \widetilde{\Sigma_g}_{(\mathcal{D})};\Q)^{\otimes r})$ is the component where $i_1=\cdots=i_r=1$. Other components can be determined as follows:
\begin{enumerate}
\item If some $i_j=0$, the component $$H^p(\M_{g,1}(\ell);H^{i_1}(\widetilde{\Sigma_g}_{(\mathcal{D})};\Q)\otimes \cdots \otimes H^{i_j}(\widetilde{\Sigma_g}_{(\mathcal{D})};\Q) \otimes \cdots \otimes H^{i_r}(\widetilde{\Sigma_g}_{(\mathcal{D})};\Q))$$ appears in $H^p (\M_{g,1};H^q((\widetilde{\Sigma_g}_{(\mathcal{D})})^{\times r};\Q))$ where $q\le2r-2$.
For $2\le i \le r+1$, let the map $\pi^{(i)}: \mathcal{C}_{g,r+1}(\ell) \to \mathcal{C}_{g,r}(\ell)$ be defined by
$$(C,x_1,y_2,\cdots,y_{r+1},h)\mapsto (C,x_1,y_2,\cdots\widehat{y_i},\cdots,y_{r+1},h).$$
By definition, the map $\pi :\mathcal{C}_{g,r+1}(\ell) \to \mathcal{M}_{g,1}(\ell)$ factors through $\pi^{(i)}$. We then have the following commuting diagram:
\begin{equation*}\begin{aligned}\xymatrix{(\widetilde{\Sigma_g}_{(\mathcal{D})})^{\times r} \ar[d] \ar[r] & \mathcal{C}_{g,r+1}(\ell) \ar[r] \ar[d]^{\pi^{(i)}}  & \mathcal{M}_{g,1}(\ell) \ar[d]^{id} \\ (\widetilde{\Sigma_g}_{(\mathcal{D})})^{\times (r-1)}\ar[r]& \mathcal{C}_{g,r}(\ell) \ar[r]& \mathcal{M}_{g,1}(\ell) }\end{aligned},\end{equation*}
which induces maps between the $E_2$ terms of the two Leray spectral sequences:
$$H^{p}(\mathcal{M}_{g,1}(\ell);H^q((\widetilde{\Sigma_g}_{(\mathcal{D})})^{\times (r-1)};\Q)) \to H^{p}(\mathcal{M}_{g,1}(\ell);H^q((\widetilde{\Sigma_g}_{(\mathcal{D})})^{\times r};\Q)), q\le 2r-2.$$
After we expand the left-hand side by the K\"unneth formula, these terms can be identified by induction.
\item If some  $i_j=2$, such a component appears as the image of the cup product
\begin{equation*}\begin{aligned}
H^{p}(\M_{g,1}(\ell);H^{i_1}(\widetilde{\Sigma_g}_{(\mathcal{D})};\Q)\otimes \cdots& \otimes H^{\widehat{i_{j}}}(\widetilde{\Sigma_g}_{(\mathcal{D})};\Q) \otimes \cdots \otimes H^{i_r}(\widetilde{\Sigma_g}_{(\mathcal{D})};\Q)) \\
\otimes H^0(\M_{g,1}(\ell)&;H^{2}(\widetilde{\Sigma_g}_{(\mathcal{D})};\Q))\\ &\Big\downarrow \\ H^p(\M_{g,1}(\ell);H^{i_1}(\widetilde{\Sigma_g}_{(\mathcal{D})};\Q)\otimes  \cdots & \otimes H^{i_j}(\widetilde{\Sigma_g}_{(\mathcal{D})};\Q) \otimes \cdots \otimes H^{i_r}(\widetilde{\Sigma_g}_{(\mathcal{D})};\Q)).
\end{aligned}
\end{equation*}
Here the term 
$$H^{p}(\M_{g,1}(\ell);H^{i_1}(\widetilde{\Sigma_g}_{(\mathcal{D})};\Q)\otimes \cdots \otimes H^{\widehat{i_{j}}}(\widetilde{\Sigma_g}_{(\mathcal{D})};\Q) \otimes \cdots \otimes H^{i_r}(\widetilde{\Sigma_g}_{(\mathcal{D})};\Q)) $$
 is known by induction. Additionally, by following essentially the same argument as in Lemma \ref{lem3.5}, but applied to the Leray spectral sequence associated with the map $\mathcal{C}_{g,2}(\ell)\to \mathcal{M}_{g,1}(\ell)$, we deduce that $H^0(\M_{g,1}(\ell);H^{2}(\widetilde{\Sigma_g}_{(\mathcal{D})};\Q))$ is isomorphic to $\Q$ generated by $a_{(\{1<j+1\},1\in \mathcal{D})}$. Using this, we can compute the image and simplify the result by employing the relations (\ref{va1}) (\ref{va2}) (\ref{va3}) satisfied by $v_i$ and $a_{\widetilde{I}}$.
\end{enumerate}
We carefully list components of these two types in Table 2 below, where polynomials in the table refer to the degree $k$ terms and $K$ denotes $H^\bullet (\M_{g}(\ell);\Q)$.  
\begin{table}
\footnotesize
\caption{Rational cohomology of $\mathcal{C}_{g,r+1}(\ell)$ written in two Ways}
\centering
\begin{tabular}{|c|c|c|}
\hline
 & & \\
$\widetilde{P}\in \mathcal{P}^{\mathcal{D}}_{r+1}$ & $H^k (\mathcal{C}_{g,r+1}(\ell);\Q)$ &  $ \bigoplus\limits_{p+q=k} H^p (\mathcal{M}_{g,1}(\ell);H^q(\widetilde{\Sigma_g}_{(\mathcal{D})}\times \widetilde{\Sigma_g}_{(\mathcal{D})};\Q))$ \\
\hline
$\{1\},\{2\},\cdots,\{r+1\}$ & $K\otimes \Q[v_1,v_2, \cdots, v_{r+1}]$ & $K\otimes\Q[u_1]$ \\
 & & $K\otimes v_a\Q[v_1,v_a],\forall a\ge 2$ \\
 & & $K\otimes v_b v_c\Q[v_1,v_b,v_c],2\le b<c$ \\
 & & \vdots \\
 & & $v_{a_1}v_{a_2}\cdots v_{a_{r-1}}[v_1, v_{a_1},v_{a_2},\cdots, v_{a_{r-1}}],$ \\
 & & $2\le  a_1<a_2<\cdots <a_{r-1}$ \\
 & & ? $\subset H^{k-r}(\mathcal{M}_{g,1}(\ell);H^1( \widetilde{\Sigma_g}_{(\mathcal{D})};\Q)^{\otimes r})$\\
\hline
 & & \\
$\{1\},I_2=(S_2,\vec d_2), \cdots$& $K\otimes \Q[v_1,v_{I_j}:j\ge 2]a_{\widetilde{P}}$ & ? $\subset H^{k-r}(\mathcal{M}_{g,1}(\ell);H^1( \widetilde{\Sigma_g}_{(\mathcal{D})};\Q)^{\otimes r})$\\
$\cdots,I_{\nu}=(S_\nu,\vec d_{\nu})$ & &\\
$|S_j|\ge 2,\forall j\ge 2$ & &\\
\hline
 & & \\
$\{1\},S_2=\{s_2\},\cdots$ &  $K\otimes  \Q[v_1, v_{s_2},\cdots$ & $K\otimes\Q[v_1,v_{I_j}:j\ge m]a_{\widetilde{P}}$\\
$\cdots,S_{m-1}=\{s_{m-1}\},$ & $\cdots,v_{s_m}, v_{I_j}:j\ge m]a_{\widetilde{P}} $ & $K\otimes v_a \Q[v_1,v_{s_a},v_{I_j}:j\ge m]a_{\widetilde{P}},$\\
$I_m=(S_m,\vec d_m), \cdots $& & $2\le a \le m-1,m>3$ \\
$\cdots,I_{\nu}=(S_\nu,\vec d_{\nu})$& &    \vdots \\
$m>2,\forall j\ge m,|S_j|\ge 2$& &   $v_{s_{a_1}}\cdots v_{s_{a_{m-3}}}\Q[v_1,v_{s_{a_1}},\cdots$ \\
 & & $\cdots, v_{s_{a_{m-3}}},v_{I_j}:j\ge m]a_{\widetilde{P}} $\\
 & & ($2\le a_1<\cdots <a_{m-3}\le m-1$.) \\
 & & ? $\subset H^{k-r}(\mathcal{M}_{g,1}(\ell);H^1( \widetilde{\Sigma_g}_{(\mathcal{D})};\Q)^{\otimes r})$\\
\hline
 & & \\
$(S_1,\vec d_1),\cdots,(S_\nu,\vec d_{\nu})$ & $K\otimes \Q[u_{I_j}:1\le j\le \nu]a_{\widetilde{P}}  $  & If for $\vec d_1=(d_1^{(1)},d_1^{(2)},\cdots, d_1^{(|S_1|-1)}),$\\
$1\in S_1,|S_1|\ge 2$ & & $\exists i$ such that $d_1^{(i)}=1:$ all can be realized outside\\
 & &  $H^{k-r}(\mathcal{M}_{g,1}(\ell);H^1( \widetilde{\Sigma_g}_{(\mathcal{D})};\Q)^{\otimes r})$;  \\
 & & if $\forall i$, $d_1^{(i)}\neq 1:$ \\
& & all except $ ( \prod\limits_{\{i\}\in \widetilde{P},i\neq 1} v_i )\Q[v_I:I \in \widetilde{P}] a_{\widetilde{P}}$ can be\\
& & realized outside $ H^{k-r}(\mathcal{M}_{g,1}(\ell);H^1( \widetilde{\Sigma_g}_{(\mathcal{D})};\Q)^{\otimes r})$.\\
\hline
\end{tabular}
\end{table}
Since all isomorphisms in the table are isomorphims of (polarized) mixed Hodge structures, by semi-simplicity, we have, when $g\ge 2k^2+7k+2$,
$$H^{\bullet-r}(\mathcal{M}_{g,1}(\ell);H^1( \widetilde{\Sigma_g}_{(\mathcal{D})};\Q)^{\otimes r})\cong H^\bullet (\M_{g}(\ell);\Q)\otimes \left( \bigoplus_{\widetilde{P}\in \mathcal{P}^{\mathcal{D}}_{r+1}(J)} ( \prod_{\{i\}\in \widetilde{P},i\neq 1} v_i )\Q[v_{\widetilde{I}}:\widetilde{I}\in \widetilde{P}] a_{\widetilde{P}}\right),$$
in degrees $k$ such that  $g\ge 2k^2+7k+2$. Here $J=(0,0,\cdots,0)$.
\end{proof}

\begin{proof}[\rm\bf{Proof of Theorem \ref{gcase} for general $J$}]
Let $J=(J_1,J_2,\cdots,J_r)$. We have just proved the case $J=(0,0,\cdots,0)$. We now prove the theorem for a general $J$ by induction on $r$ and on $\sum\limits_{i=1}^r J_i$.

When $r=0$, the theorem holds by the same reasoning as in the beginning of the previous proof.

For $r\ge1$, assume by induction that the theorem holds for cases $\le (r-1)$. When $\sum\limits_{i=1}^r J_i=0$, the proof has already been provided. Now, let $\sum\limits_{i=1}^r J_i=m>0$, and assume by induction that the theorem holds for $\sum\limits_{i=1}^r J_i=m-1$. 

For $J=(J_1,J_2,\cdots,J_r)$, since $\sum\limits_{i=1}^r J_i>0$,  there exists at least one $J_t$ that is $1$. Replacing the $t$-th entry in $J$ by $0$, we obtain a new array $\widetilde{J}$ with $\sum\limits_{i=1}^r \widetilde{J}_i=m-1$. We can then write:
$$\fH^r(J)=f(J_1)\otimes \cdots \otimes f(J_t=1)\otimes\cdots \otimes f(J_r),$$ $$\fH^r(\widetilde{J})=f(J_1)\otimes \cdots \otimes f(\widetilde{J_t}=0)\otimes\cdots \otimes f(J_r).$$
Here $f(0)=H^1(\widetilde{\Sigma_g}_{(\mathcal{D})};\Q)$ and $f(1)=\fH_{g,1}(\ell;\Q)$. Recall that the Prym representation $\fH_{g,1}(\ell;\Q)$ is $H^1(\widetilde{\Sigma_{g,1}}_{(\mathcal{D})};\Q)$, where $\widetilde{\Sigma_{g,1}}_{(\mathcal{D})}$ is the regular $\mathcal{D}$-cover of $\Sigma_{g,1}$. By filling in all punctures of $\widetilde{\Sigma_{g,1}}_{(\mathcal{D})}$, we obtain the following short exact sequence of $\M_{g,1}(\ell)$-modules:
$$0 \to H^1(\widetilde{\Sigma_g}_{(\mathcal{D})};\Q)\to \fH_{g,1}(\ell;\Q) \to \Q^{|\mathcal{D}|-1} \to 0.$$
Tensoring this short exact sequence with $f(J_1)\otimes \cdots \otimes f(J_{t-1})$ on the left and with $f(J_{t+1})\otimes\cdots \otimes f(J_r)$ on the right, we obtain:
$$0\to \fH^r(\widetilde{J})\to \fH^r(J) \to f(J_1)\otimes \cdots \otimes f(J_{t-1})\otimes \Q^{|\mathcal{D}|-1}  \otimes f(J_{t+1})\otimes\cdots \otimes f(J_r) \to 0.$$
This short exact sequence of $\M_{g,1}(\ell)$-modules induces the following long exact sequence:
$$ \begin{aligned}\cdots & \to H^{\bullet-r-1}(\M_{g,1}(\ell);f(J_1)\otimes \cdots \otimes f(J_{t-1})\otimes \Q^{|\mathcal{D}|-1}  \otimes f(J_{t+1})\otimes\cdots \otimes f(J_r))\to\\  & \to H^{\bullet-r} (\M_{g,1}(\ell);\fH^r(\widetilde{J})) \to 
 H^{\bullet-r}(\M_{g,1}(\ell); \fH^r(J))\to \\ & \to H^{\bullet-r}(\M_{g,1}(\ell);f(J_1)\otimes \cdots \otimes f(J_{t-1})\otimes \Q^{|\mathcal{D}|-1}  \otimes f(J_{t+1})\otimes\cdots \otimes f(J_r))\to\cdots.\end{aligned}$$
Our goal is to compute $ H^{\bullet-r}(\M_{g,1}(\ell); \fH^r(J))$. Let $\phi_\bullet$ denote the map \begin{equation}\label{map} \begin{aligned} &H^{\bullet-r-1}(\M_{g,1}(\ell);f(J_1)\otimes \cdots \otimes f(J_{t-1})\otimes \Q^{|\mathcal{D}|-1}  \otimes f(J_{t+1})\otimes\cdots \otimes f(J_r)) \\ \to & H^{\bullet-r} (\M_{g,1}(\ell);\fH^r(\widetilde{J})).\end{aligned}\end{equation}
Then we have
\begin{equation}\label{kercoker} 0 \to \textup{Coker}(\phi_{\bullet})\to H^{\bullet-r}(\M_{g,1}(\ell); \fH^r(J)) \to \textup{Ker}(\phi_{\bullet+1}).\end{equation}
On the one hand, the source term of the map (\ref{map}) is
$$\begin{aligned} & H^{\bullet-r-1}(\M_{g,1}(\ell);f(J_1)\otimes \cdots \otimes f(J_{t-1})\otimes \Q^{|\mathcal{D}|-1}  \otimes f(J_{t+1})\otimes\cdots \otimes f(J_r)) \\ \cong  &\bigoplus_{|\mathcal{D}|-1} H^{\bullet-r-1}(\M_{g,1}(\ell);f(J_1)\otimes \cdots \otimes f(J_{t-1})  \otimes f(J_{t+1})\otimes\cdots \otimes f(J_r)).\end{aligned}$$
Since the coefficients are $(r-1)$-tensor powers, this cohomology is known by the induction hypothesis on $r$. For each $1\neq d \in\mathcal{D}$, the component 
\begin{equation}\label{component} H^{\bullet-r-1}(\M_{g,1}(\ell);f(J_1)\otimes \cdots \otimes f(J_{t-1})  \otimes f(J_{t+1})\otimes\cdots \otimes f(J_r))\end{equation}
is isomorphic to the direct sum of 
$$\left( H^\bullet (\M_{g}(\ell);\Q)\otimes ( \prod_{\{i\}\in \widetilde{P},i\neq 1} v_i )\Q[v_{\widetilde{I}}:\widetilde{I}\in \widetilde{P}] a_{\widetilde{P}}\right)^{\bullet-2},$$
where $\widetilde{P}$ ranges over all $\mathcal{D}$-weighted partitions of the index set $\{1,\cdots,t,\widehat{t+1},t+2,\cdots,r+1\}$ that are compatible with $(J_1,\cdots,J_{t-1},J_{t+1},\cdots,J_r)$. Here $[2]$ indicates a degree shift by 2.
Inserting $(t+1)$ into the index set, we can rewrite this as a direct sum of
\begin{equation}\label{polycom} \left( H^\bullet (\M_{g}(\ell);\Q)\otimes ( \prod_{\{i\}\in \widetilde{P},i\neq 1} v_i )\Q[v_{\widetilde{I}}:\widetilde{I}\in \widetilde{P}]a_{I_1\setminus \{t+1\}} \prod\limits_{m\ge 2} a_{I_m}\right)^{\bullet-2},\end{equation}
where $\widetilde{P}=\{I_1=(S_1,\vec d_1),\cdots,I_{\nu}=(S_\nu,\vec d_{\nu})\}$ ranges over all $\mathcal{D}$-weighted partitions of the index set $[r+1]$ that satisfy
\begin{itemize}
\item $1\in S_1$ and $t+1\in S_1$. 
\item Each entry in $\vec d_1$ is not $1\in\mathcal{D}$. In particular, the entry corresponding to $t+1$ is $d$.
\item $S_1$ does not contain any $2\le a \le r+1$ such that $J_{a-1}=1$.
\end{itemize}

On the other hand, since $\sum_{i=1}^{r+1}\widetilde{J}_i=m-1<m$, the target term $H^{\bullet-r} (\M_{g,1}(\ell);\fH^r(\widetilde{J}))$ of the map (\ref{map}) is known by the induction hypothesis on $\sum_{i=1}^{r+1}\widetilde{J}_i$. We have
$$H^{\bullet-r}(\M_{g,1}(\ell);\fH^r(\widetilde{J}))\cong H^\bullet (\M_{g}(\ell);\Q)\otimes \left( \bigoplus_{\widetilde{P}\in \mathcal{P}^{\mathcal{D}}_{r+1}(\widetilde{J})} ( \prod_{\{i\}\in \widetilde{P},i\neq 1} v_i )\Q[v_{\widetilde{I}}:\widetilde{I}\in \widetilde{P}] a_{\widetilde{P}}\right),$$ 
in degrees $k$ such that $g\ge 2k^2+7k+2$.

The map (\ref{map}) restricted to the component (\ref{component}) indexed by $1\neq d\in \mathcal{D}$ is multiplication by the Poincar\'e dual $a_{(\{1<t+1\},d)}$. Hence, the image of the summand (\ref{polycom}) is 
\begin{equation}\label{product} H^\bullet (\M_{g}(\ell);\Q)\otimes ( \prod_{\{i\}\in \widetilde{P},i\neq 1} v_i )\Q[v_{\widetilde{I}}:\widetilde{I}\in \widetilde{P}]a_{I_1\setminus \{t+1\}} \prod\limits_{m\ge 2} a_{I_m}\cdot a_{(\{1<t+1\},d)}.\end{equation}
The relation (\ref{va2}) for $I_1=(S_1,\vec d_1)=(\{1<t+1\},d)\cup (I_1\setminus \{t+1\})$ tells us
$$ a_{(\{1<t+1\},d)}\cdot a_{I_1\setminus \{t+1\}}=a_{I_1}.$$
Thus (\ref{product}) can be rewritten as 
$$ \begin{aligned}& H^\bullet (\M_{g}(\ell);\Q)\otimes ( \prod_{\{i\}\in \widetilde{P},i\neq 1} v_i )\Q[v_{\widetilde{I}}:\widetilde{I}\in \widetilde{P}]a_{I_1\setminus \{t+1\}} \prod\limits_{m\ge 2} a_{I_m}\cdot a_{(\{1<t+1\},d)}\\
=&H^\bullet (\M_{g}(\ell);\Q)\otimes ( \prod_{\{i\}\in \widetilde{P},i\neq 1} v_i )\Q[v_{\widetilde{I}}:\widetilde{I}\in \widetilde{P}]a_{I_1} \prod\limits_{m\ge 2} a_{I_m}\\
=&H^\bullet (\M_{g}(\ell);\Q)\otimes ( \prod_{\{i\}\in \widetilde{P},i\neq 1} v_i )\Q[v_{\widetilde{I}}:\widetilde{I}\in \widetilde{P}]  a_{\widetilde{P}}.\end{aligned}$$
Thus the image of the map (\ref{map}) is the direct sum of (\ref{product}) as $\widetilde{P}$ ranges over $\mathcal{D}$-weighted partitions of the index set $[r+1]$ that satisfy the conditions specified above. By excluding these particular $\mathcal{D}$-weighted partitions from $ \mathcal{P}^{\mathcal{D}}_{r+1}(\widetilde{J})$, the cokernel of the map (\ref{map}) is the direct sum of 
$$H^\bullet (\M_{g}(\ell);\Q)\otimes ( \prod_{\{i\}\in \widetilde{P},i\neq 1} v_i )\Q[v_{\widetilde{I}}:\widetilde{I}\in \widetilde{P}]  a_{\widetilde{P}},$$
as $\widetilde{P}$ ranges over $\mathcal{D}$-weighted partitions of the index set $[r+1]$ that are compatible with $\widetilde{J}$ and $S_1$ does not contain $t+1$. According to the definition of compatibility (Definition \ref{compa}), these are precisely the $\mathcal{D}$-weighted partitions of the index set $[r+1]$ that are compatible with $J$, since $J_t=1$.

Since the map \eqref{map} is injective, the kernel is zero. Thus by (\ref{kercoker}) we have 
$$H^{\bullet-r}(\M_{g,1}(\ell); \fH^r(J))\cong H^\bullet (\M_{g}(\ell);\Q)\otimes \left( \bigoplus_{\widetilde{P}\in \mathcal{P}^{\mathcal{D}}_{r+1}(J)} ( \prod_{\{i\}\in \widetilde{P},i\neq 1} v_i )\Q[v_{\widetilde{I}}:\widetilde{I}\in \widetilde{P}] a_{\widetilde{P}}\right).$$
This completes the induction.
\end{proof}

Our next goal is to generalize the result from $\Sigma_{g,1}$ to any non-closed surface $\Sigma_{g,p}^b$. In order to do that, we first need the following result for $\Sigma_g^1$:

\begin{cor}\label{boundary1}
There is a graded map of $H^{\bullet}(\M_{g}^1(\ell);\mathbb{Q})$-modules:
$$H^\bullet (\M_{g}^1(\ell);\Q)\otimes \left(\bigoplus_{\widetilde{P}\in\mathcal{P}_r^{\mathcal{D}}}(\prod_{\{i\}\in \widetilde{P}} v_i) \Q[v_{\widetilde{I}}:\widetilde{I}\in \widetilde{P}]a_{\widetilde{P}}\right)\to H^{\bullet-r}(\M_{g}^1(\ell);\fH_{g}^1(\ell;\Q)^{\otimes r}),$$
which is an isomorphism in degrees $k$ such that $g\ge 2k^2+7k+2$. 
\end{cor}

\begin{proof}
By Proposition \ref{lBirman}, We have the following short exact sequence obtained by gluing a punctured disk to the boundary of $\Sigma_{g}^1$:
$$1\to \mathbb{Z}\to\M_g^1(\ell) \to\M_{g,1}(\ell)\to 1.$$
This sequence induces a Gysin sequence (Proposition \ref{gy}) for cohomology with coefficients in $\fH_{g,1}(\ell;\Q)^{\otimes r}\cong \fH_g^1(\ell;\Q)^{\otimes r}$:
$$\begin{aligned}\cdots \to &H^{\bullet-r-2}(\M_{g,1}(\ell);\fH_{g,1}(\ell;\Q)^{\otimes r}) \to H^{\bullet-r}(\M_{g,1}(\ell);\fH_{g,1}(\ell;\Q)^{\otimes r})\to \\
\to &H^{\bullet-r}(\M_g^1(\ell);\fH_g^1(\ell;\Q)^{\otimes r}) \to \\
\to & H^{\bullet-r-1}(\M_{g,1}(\ell);\fH_{g,1}(\ell;\Q)^{\otimes r}) \to H^{\bullet-r+1}(\M_{g,1}(\ell);\fH_{g,1}(\ell;\Q)^{\otimes r}) \to \cdots.\end{aligned}$$
Let $\phi_{\bullet}$ denote the boundary map: $$\phi_{\bullet}:H^{\bullet-r-2}(\M_{g,1}(\ell);\fH_{g,1}(\ell;\Q)^{\otimes r}) \to H^{\bullet-r}(\M_{g,1}(\ell);\fH_{g,1}(\ell;\Q)^{\otimes r}).$$ We then have the following short exact sequence
$$1\to \textup{Coker}(\phi_{\bullet})\to H^{\bullet-r}(\M_g^1(\ell);\fH_g^1(\ell;\Q)^{\otimes r}) \to \textup{Ker}(\phi_{\bullet+1}) \to 1.$$
Since $\phi_{\bullet}$ is multiplication by the Euler class $e_1$, it is injective. The target term of $\phi_{\bullet}$ is known from Theorem \ref{case1}, so we can determine the cokernel of $\phi_{\bullet}$:
\[\begin{gathered} H^\bullet (\M_{g}(\ell);\Q)\otimes \left(\bigoplus_{\widetilde{P}\in\mathcal{P}_r^{\mathcal{D}}}(\prod_{\{i\}\in \widetilde{P}} v_i) \Q[v_{\widetilde{I}}:\widetilde{I}\in \widetilde{P}]a_{\widetilde{P}}\right),\end{gathered}\]
in degrees $k$ such that $g \ge 2k^2+7k+2$. This is isomorphic to $H^{\bullet-r}(\M_{g}^1(\ell);\fH_{g}^1(\ell;\Q)^{\otimes r})$.
\end{proof}

To generalize our results to any con-closed surface $\Sigma_{g,p}^b$, we introduce Putman's theory on partial level-$\ell$ representations (\cite{AndyStable}) as follows.
\subsection{Putman's theory on partial level-$\ell$ representations}
Let $$\fH_{g,p}^b(\ell;\mathbb{C})=\fH_{g,p}^b(\ell;\Q)\otimes_{\Q} \mathbb{C}.$$ The finite group $\mathcal{D}\cong(\mathbb{Z}/\ell)^{2g}$ acts on $\fH_{g,p}^b(\ell;\mathbb{C})$ via deck transformations. Irreducible representations of $\mathcal{D}$ are characterized by characters. A character $\chi:\mathcal{D}\to \mathbb{C}\setminus\{0\}$ gives an irreducible representation $\mathbb{C}_{\chi}$ where $d\cdot \vec v=\chi(d)\vec v$ for $d\in\mathcal{D}$ and $\vec v\in \mathbb{C}_{\chi}$.

We denote by $\widehat{\mathcal{D}}$ the group of all characters of $\mathcal{D}$, and let $\fH_{g,p}^b(\chi)$ present the $\mathbb{C}_{\chi}$-isotypic component of $\fH_{g,p}^b(\ell;\mathbb{C})$. Thus we can decompose $\fH_{g,p}^b(\ell;\mathbb{C})$ as:
$$\fH_{g,p}^b(\ell;\mathbb{C})=\bigoplus_{\chi\in\widehat{\mathcal{D}}} \fH_{g,p}^b(\chi).$$
Since the action of $\M_{g,p}^b(\ell)$ on $\fH_{g,p}^b(\ell;\mathbb{C})$ commutes with the action of $\mathcal{D}$, each isotypic component is preserved, giving a decomposition into $\M_{g,p}^b(\ell)$-modules. Taking the $r$-tensor power of $\fH_{g,p}^b(\ell;\mathbb{C})$, we get:
\begin{equation}\label{cha1} \fH_{g,p}^b(\ell;\mathbb{C})^{\otimes r}=\bigoplus_{\chi_1,\cdots,\chi_r\in\widehat{\mathcal{D}}} \fH_{g,p}^b(\chi_1)\otimes\cdots\otimes \fH_{g,p}^b(\chi_r).\end{equation}
A subgroup $H<H_1(\Sigma_{g,p}^b;\mathbb{Z}/\ell)$ is called \textbf{symplectic} if the algebraic intersection pairing on $H_1(\Sigma_{g,p}^b;\mathbb{Z}/\ell)$ restricts to a non-degenerate pairing on $H$. Here ``non-degenerate'' means that it identifies $H$ with its dual $\textup{Hom}(H,\mathbb{Z}/\ell)$. We can write
 $H\cong (\mathbb{Z}/\ell)^{2h}$, where $h$ is called the genus of $H$. 
 
For a symplectic subgroup $H$, there is a surjective homomorphism:
\begin{equation}\label{covermap} \pi_1(\Sigma_{g,p}^b)\to H_1(\Sigma_{g,p}^b;\mathbb{Z}/\ell)= H\oplus H^{\perp}\xrightarrow{proj} H.\end{equation}
This map factors through $\mathcal{D}=H_1(\Sigma_g;\mathbb{Z}/\ell)$ since the homology classes of loops around boundary components and punctures lie in $H^{\perp}$. Thus we have a surjective homomorphism $\mathcal{D}\to H$. A character $\chi\in \widehat{\mathcal{D}}$ is said to be \textbf{compatible} with $H$ if it factors through the map $\mathcal{D}\to H$. The following lemma is a special case of Putman's Lemma 6.9 in \cite{AndyStable}:

\begin{lem}\label{compatible}
For $g > r$, given $r$ characters $\chi_1,\cdots,\chi_r\in\widehat{\mathcal{D}}$, there exists a symplectic subgroup $H$ of genus $r$ such that $\chi_1,\cdots,\chi_r$ are all compatible with $H$.
\end{lem}
\begin{proof}
We define the following group homomorphism:
$$\mu_r: H_1(\Sigma_g;\mathbb{Z})\to H_1(\Sigma_g;\mathbb{Z}/\ell) \to (\mathbb{C}\setminus\{0\})^r, x\mapsto \bar{x}\mapsto (\chi_1(\bar{x}),\cdots, \chi_r(\bar{x})).$$
Since elements in $H_1(\Sigma_g;\mathbb{Z}/\ell)\cong (\mathbb{Z}/\ell)^{2g}$ have order dividing $l$, the image of each $\chi_i$ lies in the cyclic group of $l$-th roots of unity. Thus we can rewrite $\mu_r$ as
$$\mu_r: H_1(\Sigma_g;\mathbb{Z}) \to (\mathbb{Z}/\ell)^r.$$
By Putman's Lemma 3.5 in \cite{partial}, there exists a symplectic subspace $V$ of $H_1(\Sigma_g;\mathbb{Z})$ of genus $(g-r)$ such that $\mu_r |_{V}=0$. Letting $H$ be the orthogonal complement of the image of $V$ under the map $H_1(\Sigma_g;\mathbb{Z})\to H_1(\Sigma_g;\mathbb{Z}/\ell)$, we have the required symplectic subgroup with genus $r$.

\end{proof}

Let $\chi_1,\cdots,\chi_r\in\widehat{\mathcal{D}}$ be $r$ characters that are compatible with $H$. Define
$$\fH_{g,p}^b(\underline{\chi})=\fH_{g,p}^b(\chi_1)\otimes\cdots\otimes \fH_{g,p}^b(\chi_r).$$
The \textbf{partial level-$\ell$ subgroup} of $\M_{g,p}^b$ is 
$$\M_{g,p}^b(H)=\{f\in \M_{g,p}^b|f_*:H_1(\Sigma_{g,p}^b;\mathbb{Z}/\ell)\to H_1(\Sigma_{g,p}^b;\mathbb{Z}/\ell) \text{ fixes } H \text{ pointwise.} \}.$$
Let $S_H$ denotes the regular $H$-cover of $\Sigma_{g,p}^b$ corresponding to the map \eqref{covermap}. By definition, the partial level-$\ell$ subgroup $\M_{g,p}^b(H)$ acts on $\fH_{g,p}^b(H;\mathbb{C})\defeq H^1(S_H;\mathbb{C})$. In particular, this induces an action of $\M_{g,p}^b(H)$ on $\fH_{g,p}^b(\underline{\chi})$ due to the following decomposition of  $\fH_{g,p}^b(H;\mathbb{C})$ by Putman (\cite[Lemma 6.5]{AndyStable}):
\begin{equation}\label{Hdecom}
\fH_{g,p}^b(H;\mathbb{C})\cong \bigoplus_{\chi\in \widehat{H}} \fH_{g,p}^b(\chi),
\end{equation}
where $\widehat{H}$ denotes the set of all characters of $\mathcal{D}$ that are compatible with $H$.

By definition, $\M_{g,p}^b(\ell)<\M_{g,p}^b(H)$ since $H<H^1(\Sigma_{g,p}^b;\mathbb{Z}/\ell)$. The following theorem of Putman (\cite[Theorem D]{AndyStable}) allows us to identify the twisted cohomology of $\M_{g,p}^b(\ell)$ with the twisted cohomology of $\M_{g,p}^b(H)$:
\begin{thm}[Putman \protect{\cite[Theorem D]{AndyStable}}]\label{Hisoml}
Let $g, p, b \ge 0$ and $l\ge 2$ be such that $p + b \ge 1$. Let $H$ be a symplectic subgroup of $H^1(\Sigma_g^b;\mathbb{Z}/\ell)$ and let $\chi_1,\cdots, \chi_r\in \mathcal{D}$ be $r$ characters that are compatible with $H$. Assume that $g \ge 2(k + r)^2 + 7k + 6r + 2$. Then the map
$$H^k(\M_{g,p}^b(H);\fH_{g,p}^b(\underline{\chi}))\to H^k(\M_{g,p}^b(\ell);\fH_{g,p}^b(\underline{\chi}) )$$
induced by the inclusion $\M_{g,p}^b(\ell) \hookrightarrow \M_{g,p}^b(H) $ is an isomorphism. 

\end{thm}

The following theorem of Putman (\cite[Theorem 8.1]{AndyStable}) implies a phenomenon where, when $p=0$, the twisted cohomology of $\M_{g}^b(H)$ is independent of $b$, the number of boundary components. This phenomenon is important for our later proof.

\begin{thm}[Putman \protect{\cite[Theorem 8.1]{AndyStable}}]\label{Hisom}
Let $\iota: \Sigma_g^b\to \Sigma_{g'}^{b'}$ be an orientation-preserving embedding between surfaces
with nonempty boundary. For some $l\ge 2$, let $H$ be a genus-$h$ symplectic subgroup of $H_1(\Sigma_g^b;\mathbb{Z}/\ell)$. Fix some $k, r \ge 0$, and assume that $g \ge (2h + 2)(k + r) + (4h + 2)$. Then the induced map
$$H^k(\M_{g'}^{b'}(H);\fH_{g'}^{b'}(H;\mathbb{C})^{\otimes r}) \to H^k(\M_{g}^{b}(H);\fH_{g}^{b}(H;\mathbb{C})^{\otimes r}) $$
is an isomorphism.
\end{thm}
\begin{remark}
If we take $H$ to be a symplectic subgroup of genus $g$, the partial level-$\ell$ subgroup $\M_{g,p}^b(H)$ is exactly the  level-$\ell$ subgroup $\M_{g,p}^b(\ell)$, since $\M_{g,p}^b$ acts trivially on the subgroup of $H_1(\Sigma_{g,p}^b;\mathbb{Z}/\ell)$ generated by loops around the boundary components and punctures of $\Sigma_{g,p}^b$. However, one cannot apply Theorem \ref{Hisom} directly to $\M_{g,p}^b(\ell)$, because if we let $h=g$, the condition $g \ge (2h + 2)(k + r) + (4h + 2)$ cannot hold.
\end{remark}

\subsection{Proof of Theorem \ref{main2}} We now prove the case of $\Sigma_{g,p}^b$ with $p+b\ge 1$ by induction on $p$. Recall that we aim to show
$$\begin{aligned}H^{\bullet-r}(\M_{g,p}^b(\ell);\fH_{g,p}^b(\ell;\Q)^{\otimes r})\cong H^{\bullet}(\M_{g,p}(\ell);\mathbb{Q})  \otimes
 \left(\bigoplus\limits_{\widetilde{P}\in\mathcal{P}_r^{\mathcal{D}}}(\prod\limits_{\{i\}\in \widetilde{P}} v_i) \Q[v_{\widetilde{I}}:\widetilde{I}\in \widetilde{P}]a_{\widetilde{P}}\right)\end{aligned}$$ in degrees $k$ such that $g\ge 2k^2+7k+2$. We will prove the statement over $\mathbb{C}$.

For $p=0$, the case of $b=1$ follows from Corollary \ref{boundary1}. It then suffices to show that this twisted cohomology is independent of $b$ for $b\ge 1$. Fix $b\ge 1$. Observe that there is an orientation-preserving embedding $\Sigma_g^1 \hookrightarrow \Sigma_g^b$ by gluing a surface homeomorphic to $\Sigma_0^{b+1}$ to the boundary of $\Sigma_g^1$. Then, by Theorem \ref{Hisom}, we have the following isomorphism for any symplectic subgroup $H$ of genus $h$, when $g\ge (2h+2)k+(4h+2)$:
\begin{equation}\label{b1} H^{k-r}(\M_g^b(H);\fH_g^b(H;\mathbb{C})^{\otimes r})\cong H^{k-r}(\M_g^1(H);\fH_g^1(H;\mathbb{C})^{\otimes r}).\end{equation}
These coefficients can be decomposed into direct sums of $H$-isotypic components by \eqref{Hdecom}:
$$\begin{aligned} \fH_g^1(H;\mathbb{C})^{\otimes r}\cong (\bigoplus_{\chi\in \widehat{H}} \fH_g^1(\chi))^{\otimes r}=\bigoplus_{\chi_1,\cdots,\chi_r\in\widehat{H}} \fH_g^1(\chi_1)\otimes\cdots\otimes\fH_g^1(\chi_r); \\
\fH_g^b(H;\mathbb{C})^{\otimes r}\cong (\bigoplus_{\chi\in \widehat{H}} \fH_g^b(\chi))^{\otimes r}=\bigoplus_{\chi_1,\cdots,\chi_r\in\widehat{H}} \fH_g^b(\chi_1)\otimes\cdots\otimes\fH_g^b(\chi_r).\end{aligned}$$
We can expand both sides of \eqref{b1} using the K\"unneth formula. Since the action of $\M_g^b$ (resp. $\M_g^1$) commutes with the action of $H$, we have an isomorphism in each direct sum component when $g\ge (2h+2)k+(4h+2)$:
\begin{equation}\label{cha3} H^{k-r}(\M_g^b(H);\fH_g^b(\chi_1)\otimes\cdots\otimes\fH_g^b(\chi_r))\cong H^{k-r}(\M_g^1(H);\fH_g^1(\chi_1)\otimes\cdots\otimes\fH_g^1(\chi_r)).\end{equation}
By Lemma \ref{compatible}, for each pair $\underline{\chi}=(\chi_1,\cdots,\chi_r)$, where $\chi_1,\cdots,\chi_r\in \widehat{\mathcal{D}}$, there exists a genus-$r$ symplectic subgroup, which we denote by $H_{\underline{\chi}}$, such that $\chi_1,\cdots,\chi_r$ are compatible with $H_{\underline{\chi}}$. By Theorem \ref{Hisoml}, we can identify the twisted cohomology of $\M_g^b(H_{\underline{\chi}})$ with the twisted cohomology of $\M_g^b(\ell)$, with coefficients in $\fH_g^b(\underline{\chi})=\fH_g^b(\chi_1)\otimes\cdots\otimes\fH_g^b(\chi_r)$:
\begin{equation}\label{cha2} H^{k-r}(\M_g^b(H_{\underline{\chi}});\fH_g^b(\underline{\chi}))\cong H^{k-r}(\M_g^b(\ell);\fH_g^b(\underline{\chi})),\end{equation}
if $g\ge 2k^2+7k-r+2$. Combing these facts, we have:
$$\begin{aligned}
H^{k-r}(\M_g^b(\ell);\fH_g^b(\ell;\mathbb{C})^{\otimes r}) & \cong H^{k-r}(\M_g^b(\ell); \bigoplus_{\chi_1,\cdots,\chi_r\in\widehat{\mathcal{D}}} \fH_g^b(\chi_1)\otimes\cdots\otimes\fH_g^b(\chi_r)) && \text{ by \eqref{cha1} } \\
&\cong \bigoplus_{\underline{\chi}\in(\widehat{\mathcal{D}})^{\times r}} H^{k-r}(\M_g^b(\ell);\fH_g^b(\underline{\chi}))&& \text{K\"unneth}\\
&\cong \bigoplus_{\underline{\chi}\in(\widehat{\mathcal{D}})^{\times r}} H^{k-r}(\M_g^b(H_{\underline{\chi}});\fH_g^b(\underline{\chi})) && \text{by \eqref{cha2}}\\
&\cong \bigoplus_{\underline{\chi}\in(\widehat{\mathcal{D}})^{\times r}} H^{k-r}(\M_g^1(H_{\underline{\chi}});\fH_g^1(\underline{\chi}))&& \text{by \eqref{cha3}} \\
&\cong \bigoplus_{\underline{\chi}\in(\widehat{\mathcal{D}})^{\times r}} H^{k-r}(\M_g^1(\ell);\fH_g^1(\underline{\chi}))&&\text{by \eqref{cha2}}\\
&\cong  H^{k-r}(\M_g^1(\ell); \bigoplus_{\chi_1,\cdots,\chi_r\in\widehat{\mathcal{D}}} \fH_g^1(\chi_1)\otimes\cdots\otimes\fH_g^1(\chi_r))&&\text{K\"unneth}\\
&\cong H^{k-r}(\M_g^1(\ell);\fH_g^1(\ell;\mathbb{C})^{\otimes r})&&\text{by \eqref{cha1}}
\end{aligned},$$
where we need to be a little bit careful with the range. In the places where we apply \eqref{cha2}, the range is $g\ge 2k^2+7k-r+2$. In the fourth isomorphism where we apply \eqref{cha3}, we let $h=r$, so the range is $g\ge (2h+2)k+(4h+2)=(2r+2)k+(4r+2)$. Since $k\ge r$, we have $$2k^2+7k-r+2\ge (2r+2)k+(4r+2).$$ Thus the above statement holds in the range $g\ge 2k^2+7k-r+2$. Since Corollary \ref{boundary1} for $H^{k-r}(\M_g^1(\ell);\fH_g^1(\ell;\mathbb{C})^{\otimes r})$ holds when $g\ge 2k^2+7k+2$, and $2k^2+7k+2\ge 2k^2+7k-r+2$, the theorem for $H^{k-r}(\M_g^b(\ell);\fH_g^b(\ell;\mathbb{C})^{\otimes r})$ holds when $g\ge 2k^2+7k+2$ as well.

Now, let $p\ge 1$. By Proposition \ref{lBirman}, we have the following short exact sequence obtained by gluing a punctured disk to $\Sigma_{g,p-1}^{b+1}$:
$$1\to \mathbb{Z}\to\M_{g,p-1}^{b+1}(\ell)\to \M_{g,p}^b(\ell)\to 1.$$
This sequence induces a Gysin sequence (Proposition \ref{gy}) with coefficients in $\fH_{g,p-1}^{b+1}(\ell;\Q)^{\otimes r}\cong \fH_{g,p}^b(\ell;\Q)^{\otimes r}$:
$$\begin{aligned}
\cdots \to& H^{\bullet-r}(\M_{g,p}^b(\ell);\fH_{g,p}^b(\ell;\Q)^{\otimes r})\to H^{\bullet-r}(\M_{g,p-1}^{b+1}(\ell);\fH_{g,p-1}^{b+1}(\ell;\Q)^{\otimes r}) \to \\ \to & H^{\bullet-r-1}(\M_{g,p}^b(\ell);\fH_{g,p}^b(\ell;\Q)^{\otimes r}) \to H^{\bullet-r+1}(\M_{g,p}^b(\ell);\fH_{g,p}^b(\ell;\Q)^{\otimes r}) \to \cdots.
\end{aligned}$$
Here the map $$\phi_{\bullet-r-1}:H^{\bullet-r-1}(\M_{g,p}^b(\ell);\fH_{g,p}^b(\ell;\Q)^{\otimes r}) \to H^{\bullet-r+1}(\M_{g,p}^b(\ell);\fH_{g,p}^b(\ell;\Q)^{\otimes r})$$ is multiplication by the Euler class $e_p\in H^2(\M_{g,p}^b(\ell);\Q)$, so it is injective. Similar to what we did in Corollary \ref{boundary1}, by observing the short exact sequence
$$1\to \textup{Coker}(\phi_{\bullet-r-2}) \to H^{\bullet-r}(\M_{g,p-1}^{b+1}(\ell);\fH_{g,p-1}^{b+1}(\ell;\Q)^{\otimes r})  \to \textup{Ker}(\phi_{\bullet-r-1})\to 1,$$
we get
$$H^{\bullet}(\M_{g,p}^b(\ell);\fH_{g,p}^b(\ell;\Q)^{\otimes r})\cong H^{\bullet}(\M_{g,p-1}^{b+1}(\ell);\fH_{g,p-1}^{b+1}(\ell;\Q)^{\otimes r})\otimes_{\Q} \Q[e_p].$$
We know $H^{\bullet}(\M_{g,p-1}^{b+1}(\ell);\fH_{g,p-1}^{b+1}(\ell;\Q)^{\otimes r})$ by the induction hypothesis. Thus we conclude 
$$\begin{aligned}& H^{\bullet-r}(\M_{g,p}^b(\ell);\fH_{g,p}^b(\ell;\Q)^{\otimes r})\\ \cong& H^{\bullet}(\M_{g,p-1}(\ell);\mathbb{Q})\otimes \Q[e_p]
\otimes \left( \bigoplus\limits_{\widetilde{P}\in\mathcal{P}_r^{\mathcal{D}}}(\prod_{\{i\}\in \widetilde{P}} v_i) \Q[v_{\widetilde{I}}:\widetilde{I}\in \widetilde{P}]a_{\widetilde{P}}\right)\\
\cong & H^{\bullet}(\M_{g,p}(\ell);\mathbb{Q})
\otimes \left( \bigoplus\limits_{\widetilde{P}\in\mathcal{P}_r^{\mathcal{D}}}(\prod_{\{i\}\in \widetilde{P}} v_i) \Q[v_{\widetilde{I}}:\widetilde{I}\in \widetilde{P}]a_{\widetilde{P}}\right)\end{aligned}$$
in degrees $k$ such that $g\ge 2k^2+7k+2$. Note that this twisted cohomology is independent of $b$ (the number of boundary components) but depends on $p$ (the number of punctures) and $g$ (the genus).\hfill $\qed$

\section{Infinitesimal Rigidity of Symplectic Prym Representations}

In this section, we aim to prove Theorem \ref{main1} about infinitesimal rigidity of symplectic Prym representations.

Recall that the symplectic Prym representation is defined in the following way. Let $\widetilde{S}\to S$ be a finite-abelian cover with deck group $A$. Let $\M(S,A)$ be the subgroup of $\M(S)$ fixing $A$ pointwise. Denote by $\widehat{S}$ the closed surface obtained by gluing disks to all boundary components and filling in all punctures of $\widetilde{S}$. Then there is a symplectic action of $\M(S,A)$ on $H^1(\widehat{S};\mathbb{R})$ which commutes with the deck group $A$:

$$\Phi: \M(S,A)\to \text{Aut}(H^1(\widehat{S};\mathbb{R}))^A.$$

Denote the target Lie group by $G_A$. Let $h$ be the genus of $\widehat{S}$, then we have

 $$G_A=Sp(2h;\mathbb{R})^{A}=\{P\in Sp(2h;\mathbb{R})| Pa=aP, \forall a\in A\}.$$

First of all, we compute the Lie algebra $\mathfrak{g}_A$ of $G_A$:
\begin{lem}\label{liealgebra}
The Lie algebra $\mathfrak{g}_A$ of $G_A=Sp(2h;\mathbb{R}))^{A}$ is $$\mathfrak{g}_A=\mathfrak{sp}(2h;\mathbb{R})^{A}=\{X\in \mathfrak{sp}(2h;\mathbb{R})| Xa=aX, \forall a\in A\}.$$
\end{lem}
\begin{proof}
The Lie group $G_A$ can be described as
$$G_A=\{P\in GL(2h;\mathbb{R})|P^{T}JP=J, Pa=aP, \forall a\in A\},$$
where $$J=\left(\begin{matrix} 0& I_h\\-I_h & 0\end{matrix}\right).$$
The Lie algebra $\mathfrak{g}_A$ of the matrix Lie group $G_A$ is given by:
$$\mathfrak{g}_A=\{X\in Mat(2h;\mathbb{R})|e^{tX}\in G_A, \forall t\in \mathbb{R}\}.$$
The Lie algebra of $Sp(2h;\mathbb{R})$, denoted by $\mathfrak{sp}(2h;\mathbb{R})$, satisfies $X^{T}J+JX=0$. It remains to determine how the condition $Pa=aP,\forall a\in A$ descends to $\mathfrak{g}_A$.

For $X\in \mathfrak{g}_A$, we require that $e^{tX}a=a e^{tX}$ for any $a\in A$. Expanding $e^{tX}$ as a power series $e^{tX}=\sum\limits_{j=o}^{g}\frac{(tX)^j}{j!}$, and disregarding terms of order $t^2$ or higher, we obtain
$tXa=atX$ for all $t\in \mathbb{R}$. Thus we have $Xa=aX$ for any $a \in A$, so we conclude that $\mathfrak{g}_A=\mathfrak{sp}(2h;\mathbb{R})^{\mathcal{D}}$.
\end{proof}

Notice that by composing $\Phi: \M(S,A)\to G_A$ with the adjoint representation $\textup{Ad}:G_A\to \textup{GL}(\mathfrak{g}_A)$, we can view the Lie algebra $\mathfrak{g}_A$ is a $\M(S,A)$-module. We observe:
 
\begin{lem}\label{submod}
The Lie algebra $\mathfrak{g}_A$ is a $\M(S,A)$-submodule of $H^1(\widehat{S};\mathbb{R})^{\otimes2}\cong (\mathbb{R}^{2h})^{\otimes 2}$.
\end{lem}
\begin{proof}
We first embed $\mathfrak{g}_A$ into $H^1(\widehat{S};\mathbb{R})^{\otimes2}\cong (\mathbb{R}^{2h})^{\otimes 2}$ as follows:
$$\mathfrak{g}_A=\mathfrak{sp}(2h;\mathbb{R})^{A}\subset \text{Mat}(2h;\mathbb{R})\cong (\mathbb{R}^{2h})^* \otimes \mathbb{R}^{2h}\cong  (\mathbb{R}^{2h})^{\otimes 2},$$
where the isomorphism $(\mathbb{R}^{2h})^* \cong  \mathbb{R}^{2h}$ is induced by the nondegenerate algebraic intersection form $i:H_1(\widehat{S};\mathbb{R})\times H_1(\widehat{S};\mathbb{R})  \to \mathbb{R}$.

It remains to show that the action of $\M(S,A)$ on $\mathfrak{g}_A$ is compatible with its action on $H^1(\widehat{S};\mathbb{R})^{\otimes2}$.

Take a symplectic basis $\{\alpha_1,\beta_1,\cdots,\alpha_h,\beta_h\}$ of $H^1(\widehat{S};\mathbb{R})$. For $X\in \mathfrak{g}_A$, we can write $$\begin{aligned}X&=\sum\limits_{j=1}^h (\alpha_j)^*\otimes X\alpha_j+\sum\limits_{j=1}^h (\beta_j)^*\otimes X\beta_j\in  (\mathbb{R}^{2h})^* \otimes \mathbb{R}^{2h} \\
&=\sum\limits_{j=1}^h \beta_j\otimes X\alpha_j+\sum\limits_{j=1}^h (-\alpha_j)\otimes X\beta_j \in (\mathbb{R}^{2h})^{\otimes 2}
 .\end{aligned}$$
The action of $f\in \M(S,A)$ on $X\in \mathfrak{g}_A$ is defined by $f\cdot X=FXF^{-1}$, where $F=\Phi(f)\in Sp(2h;\mathbb{R})$. Thus can express $f\cdot X$ as:
$$\begin{aligned}f\cdot X&=\sum\limits_{j=1}^h (\alpha_j)^*\otimes FXF^{-1}\alpha_j+\sum\limits_{j=1}^h (\beta_j)^*\otimes FXF^{-1}\beta_j\in  (\mathbb{R}^{2h})^* \otimes \mathbb{R}^{2h}  \\
&=\sum\limits_{j=1}^h \beta_j\otimes FXF^{-1}\alpha_j+\sum\limits_{j=1}^h (-\alpha_j)\otimes FXF^{-1}\beta_j \in (\mathbb{R}^{2h})^{\otimes 2} \\ 
&=\sum\limits_{j=1}^h F(F^{-1}\beta_j)\otimes FX(F^{-1}\alpha_j)+\sum\limits_{j=1}^h F(-F^{-1}\alpha_j)\otimes FX(F^{-1}\beta_j) \in (\mathbb{R}^{2h})^{\otimes 2}
.\end{aligned}$$
Since $\{ F^{-1}(\alpha_1),F^{-1}(\beta_1),\cdots,F^{-1}(\alpha_h),F^{-1}(\beta_h)\}$ is also a symplectic basis of $H^1(\widehat{S};\mathbb{R})$, we can rewrite $f\cdot X$ as:
$$ \sum\limits_{j=1}^h F(\beta_j)\otimes F(X\alpha_j)+\sum\limits_{j=1}^h F(-\alpha_j)\otimes F(X\beta_j).$$
This expression confirms that the action of $\M(S,A)$ on $\mathfrak{g}_A$ coincides with its action on $H^1(\widehat{S};\mathbb{R})^{\otimes2}$.
\end{proof}

Next we prove the degree-2 twisted cohomology group of $\M(S,A)$ with coefficients in $H^1(\widehat{S};\mathbb{R})^{\otimes2}$ is $0$ for sufficiently large $g$:

\begin{lem}\label{cor1}
Let $g,p,b$ be integers such that $p+b\ge1$. When $g\ge 41$, we have
$$H^1(\M(S,A);H^1(\widehat{S};\mathbb{R})^{\otimes2})=0.$$
\end{lem}
\begin{proof}
\textbf{Step 1}: We first prove the result for $A=H_1(\Sigma_g;\mathbb{Z}/\ell)$. Note that by definition $\M(S,A)$ is usually larger than the level-$\ell$ mapping class group $\M_{g,p}^b(\ell)$ which acts trivially on $H_1(\Sigma_{g,p}^b;\mathbb{Z}/\ell)$. 

In this case, the Prym representation $\fH_{g,p}^b(\ell;\Q)$ is $H^1(\widetilde{S};\Q)$. By filling in all punctures and gluing disks to all boundary components of $\widetilde{S}$, we obtain the following short exact sequence:
\begin{equation}\label{fill} 0 \to H^1(\widehat{S};\Q)\to \fH_{g,p}^b(\ell;\Q) \to \Q^{(p+b)\cdot|A|-1} \to 0.\end{equation}
We then compute the following twisted cohomology groups one by one:
\begin{enumerate}

\item Recall that we computed $H^{k-r}(\M_{g,p}^b(\ell);\fH_{g,p}^b(\ell;\Q)^{\otimes r})$ in Theorem \ref{main2} in the range $g\ge 2k^2+7k+2$, and one feature of this twisted cohomology is that it is $0$ when $k$ is odd, since the left-hand side of the isomorphism in Theorem \ref{main2} consists of even-degree terms. In particular, setting $k=3$ and $r=2$, we get
$$H^1(\M_{g,p}^b(\ell);\fH_{g,p}^b(\ell;\Q)^{\otimes 2})=0, \text{ if } g\ge 41.$$
Similarly, for $k=3$ and $r=1$, we have
$$H^2(\M_{g,p}^b(\ell);\fH_{g,p}^b(\ell;\Q))=0, \text{ if } g\ge 41.$$
\item The short exact sequence (\ref{fill}) of $\M_{g,p}^b(\ell)$-modules induces the following long exact sequence of twisted cohomology groups:
 $$\to H^3(\M_{g,p}^b(\ell); \Q^{(p+b)\cdot|A|-1} )\to H^2(\M_{g,p}^b(\ell); H^1(\widehat{S};\Q))\to H^2(\M_{g,p}^b(\ell);\fH_{g,p}^b(\ell;\Q))\to.$$
When $g\ge 41$, the term $H^3(\M_{g,p}^b(\ell); \Q^{(p+b)\cdot|A|-1} )$ is $0$ by Putman's Theorem \ref{andy0}, and $H^2(\M_{g,p}^b(\ell);\fH_{g,p}^b(\ell;\Q))$ is also $0$ as discussed in (1). Thus we have
$$ H^2(\M_{g,p}^b(\ell); H^1(\widehat{S};\Q))=0, \text{ if } g\ge 41.$$
\item By tensoring the above short exact sequence (\ref{fill}) with $\fH_{g,p}^b(\ell;\Q)$ on the left, we obtain another short exact sequence:
$$0 \to \fH_{g,p}^b(\ell;\Q)\otimes H^1(\widehat{S};\Q)\to \fH_{g,p}^b(\ell;\Q)^{\otimes 2} \to \fH_{g,p}^b(\ell;\Q)\otimes \Q^{(p+b)\cdot|A|-1}\to 0.$$
This short exact sequence of $\M_{g,p}^b(\ell)$-modules induces a long exact sequence of twisted cohomology groups:
$$\begin{aligned} & \to H^2(\M_{g,p}^b(\ell);\fH_{g,p}^b(\ell;\Q)\otimes \Q^{(p+b)\cdot|A|-1})\to H^1(\M_{g,p}^b(\ell);\fH_{g,p}^b(\ell;\Q)\otimes H^1(\widehat{S};\Q))\to \\
& \to H^1(\M_{g,p}^b(\ell);\fH_{g,p}^b(\ell;\Q)^{\otimes 2})\to \cdots\end{aligned}$$
Here $H^2(\M_{g,p}^b(\ell);\fH_{g,p}^b(\ell;\Q)\otimes \Q^{(p+b)\cdot|A|-1})\cong \bigoplus\limits_{(p+b)\cdot|A|-1} H^2(\M_{g,p}^b(\ell);\fH_{g,p}^b(\ell;\Q))$ and $H^1(\M_{g,p}^b(\ell);\fH_{g,p}^b(\ell;\Q)^{\otimes 2})$ are both 0 when $g\ge 41$ as discussed in (1), so 
$$H^1(\M_{g,p}^b(\ell);\fH_{g,p}^b(\ell;\Q)\otimes H^1(\widehat{S};\Q))=0, \text{ if } g\ge 41.$$
\item By tensoring the above short exact sequence (\ref{fill}) with $H^1(\widehat{S};\Q)$ on the right, we obtain the following short exact sequence:
$$0 \to H^1(\widehat{S};\Q)^{\otimes 2}\to \fH_{g,p}^b(\ell;\Q)\otimes H^1(\widehat{S};\Q) \to \Q^{(p+b)\cdot|A|-1}\otimes H^1(\widehat{S};\Q)  \to 0.$$
This short exact sequence of $\M_{g,p}^b(\ell)$-modules induces a long exact sequence of twisted cohomology groups:
$$\begin{aligned} \to& H^2(\M_{g,p}^b(\ell);\Q^{(p+b)\cdot|A|-1}\otimes H^1(\widehat{S};\Q) )\to H^1(\M_{g,p}^b(\ell);H^1(\widehat{S};\Q)^{\otimes 2})\to \\
& \to H^1(\M_{g,p}^b(\ell);\fH_{g,p}^b(\ell;\Q)\otimes H^1(\widehat{S};\Q))\to\cdots \end{aligned}$$
Here the first term is $0$ when $g\ge 41$ as discussed in (2), and the third term is $0$ when $g\ge 41$ as discussed in (3). Thus we have
$$H^1(\M_{g,p}^b(\ell);H^1(\widehat{S};\Q)^{\otimes 2})=0, \text{ if } g\ge 41.$$
\item Consider the finite-index subgroup $\M_{g,p}^b(\ell)$ of $\M(S,A)$. By Proposition \ref{transfer}, the associated transfer map
$$H^1(\M_{g,p}^b(\ell);H^1(\widehat{S};\mathbb{R})^{\otimes2})\to H^1(\M(S,A);H^1(\widehat{S};\mathbb{R})^{\otimes2})$$ is surjective. Thus from (4) we have
$$H^1(\M(S,A);H^1(\widehat{S};\Q)^{\otimes 2})=0, \text{ if } g\ge 41.$$
Tensoring the result with $\mathbb{R}$, we finish the proof for $A=H_1(\Sigma_g;\mathbb{Z}/\ell)$.
\end{enumerate}
\textbf{Step 2.} We now extend the proof to a general finite abelian group $A$.

Recall that $\widetilde{S}$ denotes the regular cover of $S$ corresponding to the homomorphism 
$$\pi_1(S)\to A.$$ Since $A$ is abelian, this map factors through $\pi_1(S)\to H_1(S;\mathbb{Z})$. Letting $\omega=|A|$, this map furthermore factors through $\pi_1(S)\to H_1(S;\mathbb{Z}/\omega) $. Therefore $\M_{g,p}^b(\omega)$ is a finite-index subgroup of $\M(S,A)$. 

Denote by $\widetilde{S_{\omega}}$ the regular cover of $\Sigma_{g}$ corresponding to the map $\pi_1(\Sigma_{g})\to H_1(\Sigma_{g};\mathbb{Z}/\omega).$ By part (4) in Step 1, we have
$$H^1(\M_{g,p}^b(\omega);H^1(\widetilde{S_{\omega}};\mathbb{R})^{\otimes 2})=0, \text{ if } g\ge 41.$$
By filling in all punctures and gluing disks to all boundary components in the cover $\widetilde{S}\to S$, we obtain a cover $\widehat{S}\to \Sigma_g$, where the deck group is a quotient of $H_1(\Sigma_g;\mathbb{Z}/\omega)$. Therefore $\widetilde{S_{\omega}}$ is a finite cover of $\widehat{S}$. 

Since the action of $\M_{g,p}^b(\omega)$ on $\widehat{S_{\omega}}$ commutes with the deck group of $\widetilde{S_{\omega}}\to\widehat{\widetilde{S}}$, Maschke’s Theorem implies that $H^1(\widehat{S};\mathbb{R})^{\otimes 2}$ is a direct summand of $H^1(\widetilde{S_{\omega}};\mathbb{R})^{\otimes2}$ as an $\M_{g,p}^b(\omega)$-module. Consequently,
$$H^1(\M_{g,p}^b(\omega);H^1(\widehat{S};\mathbb{R})^{\otimes2})=0, \text{ if } g\ge 41.$$
Finally, the transfer map associated with the finite-index subgroup $\M_{g,p}^b(\omega)<\M(S,A)$, with coefficients in $H^1(\widehat{S};\Q)^{\otimes 2}$,
$$H^1(\M_{g,p}^b(\omega);H^1(\widehat{S};\mathbb{R})^{\otimes2})\to H^1(\M(S,A);H^1(\widehat{S};\mathbb{R})^{\otimes2})$$ is surjective by Proposition \ref{transfer}. Thus we conclude that
\[H^1(\M(S,A);H^1(\widehat{S};\mathbb{R})^{\otimes2})=0, \text{ if } g\ge 41.\qedhere\]
\end{proof}

Now we combine the above lemmas to prove Theorem \ref{main1}:
\begin{proof}[\rm\bf{Proof of Theorem \ref{main1}}] To prove the infinitesimal rigidity of the symplectic Prym representation:
$$\Phi:\M(S,A) \to G_A,$$
we need to show that
$$H^1(\M(S,A); \mathfrak{g}_A)=0 \text{ if }g\ge 41,$$
where $\mathfrak{g}_A$ is the Lie algebra of the Lie group $G_A$.

From Lemma \ref{liealgebra}, we have $\mathfrak{g}_A=\mathfrak{sp}(2h;\mathbb{R})^{A}$. By Lemma \ref{submod}, we see $\mathfrak{sp}(2h;\mathbb{R})^{A}$ is an $\M(S,A)$-submodule of $H^1(\widehat{S};\mathbb{R})^{\otimes2}$. Therefore, it suffices to prove
$$H^1(\M(S,A);H^1(\widehat{S};\mathbb{R})^{\otimes2})=0, \text{ if } g\ge 41,$$
which is precisely the result of Lemma \ref{cor1}. This completes the proof.
\end{proof}

\nocite{1} 
\nocite{*} 
\bibliographystyle{plain}
\bibliography{ref5}

\end{document}